\documentclass[11pt]{amsart}

\numberwithin{equation}{section}
\numberwithin{figure}{section}

\usepackage{tikz-cd}
\usepackage{amssymb}
\usepackage[T1]{fontenc}
\usepackage{lmodern}
\usepackage{mathtools}
\usepackage[headings]{fullpage}
\usepackage{tikz}
\usepackage{xparse}
\usepackage{color}
\usepackage[normalem]{ulem}
\usepackage[colorlinks=true,linkcolor=blue, citecolor=blue, urlcolor=blue,
pagebackref=true]{hyperref}
\usepackage{bookmark}
\usepackage{enumitem}
\usepackage{units}
\usepackage{ytableau}
\usepackage{varwidth}

\let\oldcenter\center
\let\oldendcenter\endcenter
\renewenvironment{center}{\setlength\topsep{4pt}\oldcenter}{\oldendcenter}

\newcount\tableauRow
\newcount\tableauCol

\newcommand\smdots{\raisebox{2px}{\makebox[1em][c]{\tiny.\kern.05\fontdimen3\font.\kern.05\fontdimen3\font.}}}
\makeatletter
\def\theenumi{\@alph\c@enumi}
\DeclareMathSizes{\@xipt}{\@xipt}{7}{6}
\makeatother

\theoremstyle{plain}
\newtheorem{theorem}[equation]{Theorem}
\newtheorem{lemma}[equation]{Lemma}
\newtheorem{corollary}[equation]{Corollary}
\newtheorem{proposition}[equation]{Proposition}
\theoremstyle{definition}

\newtheorem{remark}[equation]{Remark}

\newtheorem{example}[equation]{Example}

\newtheorem{definition}[equation]{Definition}

\newtheorem{notation}[equation]{Notation}

\newtheorem{discussion}[equation]{Discussion}

\newtheorem{observation}[equation]{Observation}

\newtheorem{construction}[equation]{Construction}

\newcommand{\MC}{\mathit{M}\kern -0.1em\mathrm{1}}
\newcommand{\MCC}{\mathit{M}\kern -0.1em\mathrm{2}}
\newcommand{\Grass}{\mathrm{G}}

\newcommand{\pl}{\mathrm{p}}
\newcommand{\Bl}{\mathrm{Block}}

\newcommand{\He}{\mathrm{Head}}
\newcommand{\GL}{\mathrm{GL}}
\newcommand{\SL}{\mathrm{SL}}

\newcommand\modp[1]{\textrm{ mod }#1}
\newcommand{\tab}{\mathbb{T}}
\newcommand{\ssyt}{\mathbb{T}}
\newcommand\sch[1]{s_{#1}}
\newcommand\lr[3]{c_{#1,#2}^{#3}}

\title[A classification of spherical Schubert varieties in the Grassmannian]{A classification of spherical Schubert varieties in the Grassmannian}
\author{Reuven Hodges}
\address{Northeastern University, Boston, Massachusetts. USA.}
\email{hodges.r@husky.neu.edu}

\author{Venkatramani Lakshmibai}
\address{Northeastern University, Boston, Massachusetts. USA.}
\email{lakshmibai@neu.edu}

\begin{document}

\begin{abstract}
 Let $L$ be a Levi subgroup of $GL_N$ which acts by left multiplication on a Schubert variety $X(w)$ in the Grassmannian $G_{d,N}$. We say that $X(w)$ is a spherical Schubert variety if $X(w)$ is a spherical variety for the action of $L$. In earlier work we provide a combinatorial description of the decomposition of the homogeneous coordinate ring of $X(w)$ into irreducible $L$-modules for the induced action of $L$. In this work we classify those decompositions into irreducible $L$-modules that are multiplicity-free. This is then applied towards giving a complete classification of the spherical Schubert varieties in the Grassmannian. 
\end{abstract}
\maketitle

\section{Introduction}
\label{sec:intro}
For a reductive group $G$, a normal $G$-variety $X$ is called a spherical variety if it has an open, dense $B$-orbit for a Borel subgroup $B$ of $G$. The spherical variety $X$, having a single open $B$-orbit, will also have a single open $G$-orbit of the form $G/H$, where $H$ is an algebraic subgroup of $G$. Such a subgroup is called a spherical subgroup, and in \cite{MR1896179}, Luna proposed a program to classify spherical subgroups of reductive groups in terms of combinatorial data that he termed the homogeneous spherical data. This classification has been completed, due to the contributions of many authors, see for example \cite{2009arXiv0907.2852C,MR3198836,MR3473657,MR2495078}. In earlier work, Luna and Vust classified the spherical embeddings of $G/H$, that is, embeddings of $G/H$ into a spherical variety such that $G/H$ is the open $G$-orbit, in terms of colored fans~\cite{MR705534}. 

The above results combine to give a complete classification of spherical varieties, but there are still many open questions in this setting. One such question is, what geometric properties of a spherical variety can be inferred by studying the associated spherical data, that is, the colored fan and homogeneous spherical data. One practical method of pursuing this question is to consider other well understood classes of varieties and ask under what conditions will they be spherical varieties.

With this in mind, let $Q$ be a parabolic subgroup of $G$. Denote by $W$ the Weyl group of $G$ and $W_Q$ the subgroup of $W$ corresponding to $Q$. Then $W^Q$ is defined to be the subset of minimal length right coset representatives of $W_Q$ in $W$. There is a natural action of $G$ on $G/Q$ given by left multiplication. For a $w \in W^Q$ we define the Schubert variety $X(w)$ to be the Zariski closure of the $B$-orbit of $wQ/Q$ in $G/Q$. These Schubert varieties will be stable under the action of certain parabolic subgroups $P$ of $G$, and hence $L$-stable for the reductive Levi subgroup $L$ of $P$. Two natural questions arise.

\begin{enumerate}[before=\setlength{\baselineskip}{5mm}, label=(\arabic*)]
\item Given a Schubert variety $X(w)$ in $G/Q$ that is $L$-stable, when is $X(w)$ a spherical $L$-variety?
\item If $X(w)$ is a spherical $L$-variety, what is the associated spherical data?
\end{enumerate} 
As the geometry of Schubert varieties is particularly well understood the answer to these questions would provide ample test cases for the project of inferring geometric properties of spherical varieties in terms of their spherical data. This paper provides a complete answer to the first question when $G/Q$ is the Grassmannian variety in type A. The first author along with M. Bilen Can explores question (1) for an arbitrary $G$ and $Q$ and shows that smooth Schubert varieties are always spherical varieties~\cite{2018arXiv180305515B}. Both authors, joint with M. Bilen Can, explore question (2) in~\cite{2018arXiv180704879B}. Further, in ~\cite{2018arXiv180704879B} the toroidal Schubert varieties in the Grassmannian are characterized, and in type A, the $\GL_p \times \GL_q$-spherical Schubert varieties are also studied.

We now give an outline of the results in this paper. A Levi-Schubert quadruple is defined to be the data $(w,d,N,L)$ where $X(w)$ is a Schubert variety in the Grassmannian $G_{d,N}$ of $d$-dimensional subspaces of $\mathbb{C}^N$ and $L$ is a Levi subgroup of $\GL_N$. Such a Levi-Schubert quadruple is called stable if the Schubert variety $X(w)$ is $L$-stable, and a stable quadruple is called spherical if $X(w)$ is a spherical $L$-variety. Our first step in classifying the spherical Levi-Schubert quadruples is to define the reduction of $(w,d,N,L)$, which is also a Levi-Schubert quadruple and is denoted $(\overline{w}, \overline{d}, \overline{N}, \overline{L})$. We then show that $(w,d,N,L)$ is spherical if and only if $(\overline{w}, \overline{d}, \overline{N}, \overline{L})$ is spherical. This reduction step makes the classification considerably simpler to state.

In a previous paper the authors give a combinatorial description of the decomposition of the homogeneous coordinate ring $\mathbb{C}[X(w)]$, for the Pl\"{u}cker embedding, into irreducible $L$-modules for the induced action of $L$~\cite{HodgesLakshmibai}. For a stable Levi-Schubert quadruple $(w,d,N,L)$ we say that it is multiplicity free if the decomposition of $\mathbb{C}[X(w)]$ into irreducible $L$-modules is multiplicity free. In Proposition~\ref{p:multFreeSphericalEquiv} it is shown that a stable $(w,d,N,L)$ is multiplicity free if and only if it is spherical.

In \cite{HodgesLakshmibai} the decomposition of $\mathbb{C}[X(w)$ into irreducible $L$-modules is given in two steps. In the first step, each degree piece of $\mathbb{C}[X(w)]$ is decomposed into simpler submodules. The second step then shows that these submodules are isomorphic to certain tensor products of skew Schur-Weyl modules (which can be easily decomposed into irreducible $L$-modules using the Littlewood-Richardson coefficients). In Proposition~\ref{p:MultFreeHighestLevel} we define two criteria $\MC$ and $\MCC$ that are stated in terms of the simpler submodules from the first step above. The proposition states that a Levi-Schubert quadruple is multiplicity free if and only if both $\MC$ and $\MCC$ are satisfied.

If $X(w)$ is a Schubert variety in $G_{d,N}$ then $w$ can be represented by the sequence $(\ell_1,\ldots,\ell_d)$ for some integers $1 \leq \ell_1 < \ell_2 < \cdots < \ell_d \leq N$ (see Section~\ref{subsec:SMT}). A Levi subgroup $L$ of $\GL_N$ is of the form $\GL_{N_1} \times \cdots \times \GL_{N_{b_L}}$ for some positive integers $b_L$ and $N_k$ with $1 \leq k \leq b_L$. Using $L$ we define a partition of $\{ 1, \ldots , N \}$ into subsets of consecutive integers denoted $\Bl_{L,k}$ for $1 \leq k \leq b_L$ (see Section~\ref{sec:Decomp}). Then the non-negative integers $h_1,\ldots,h_{b_L}$ are defined by $h_k = |\left\{j | \ell_j \in \Bl_{L, k} \right\}|$. 

In Propositions~\ref{p:MCCSatCrit} and \ref{p:MCSatCrit} we provide the exact combinatorial requirements on $b_L$, $h_1,\ldots,h_{b_L}$ and $N_1,\ldots,N_{b_L}$ such that $\MCC$ and $\MC$ are, respectively, satisfied. This allows us to prove our primary result. 

\begin{theorem}
\label{t:mainSphericalClassification}
The stable, reduced Levi-Schubert quadruple $(w,d,N,L)$ is multiplicity free (equivalently spherical) if and only if one of the following holds

\begin{enumerate}[label=(\roman*)]
\item $b_L \leq 2$
\item $b_L = 3$, and at least one of $N_2 = 1$, $h_1 + 1 \geq N_1$, $N_2=h_2$ with $h_1 + 2 \geq N_1$, $h_2 > 0$ with $h_3 < 2$, $h_2 = 0$ with $h_3 \leq 2$ holds
\item $b_L \geq 4$, $p_w = 2$ or if $p_w > 2$, then $h_1 + \cdots + h_{p_w - 1} + 1 \geq N_1 + \cdots + N_{p_w - 1}$
\end{enumerate}
where $1 < p_w < b_L - 1$ is the minimum index such that $h_{p_w + 1} + \cdots + h_{b_L} < 2$. Such an index may not exist, if it does not set $p_w = b_L - 1$.
\end{theorem}

We give even simpler criterion for a Levi-Schubert quadruple $(w,d,N,L)$ to be spherical in the case when $L$ is the maximal Levi subgroup which acts on $X(w)$ by left multiplication in Corollary~\ref{c:mainSphericalClassification}. As a nice application of this classification theorem, in Corollary~\ref{c:toricSchubert}, we give a description of the Schubert varieties in the Grassmannian that are toric varieties for a quotient of the maximal torus under the left multiplication action. 

This paper is organized into the following sections. In Section 2 the background and notation for spherical varieties, Schubert varieties, skew Schur functions, and skew Schur-Weyl modules is covered. Additionally, a few minor technical lemma involving Littlewood-Richardson coefficients are proved. Section 3 recalls the results and notation for the decomposition of the homogeneous coordinate ring of $X(w)$ into irreducible $L$-modules that was developed in \cite{HodgesLakshmibai}. Levi-Schubert quadruples and their reduction is covered in Section 4. The classification of the reduced, stable Levi-Schubert quadruples that are spherical is accomplished in Section 5. An application of these results to toric Schubert varieties is briefly discussed in Section 6.
\section{Preliminaries}
\label{sec:prelim}
\subsection{Spherical varieties}
\label{subsec:spherical}
Let $X$ be a normal $G$-variety for a reductive group $G$. The most common characterization given for $X$ to be a \emph{spherical variety} is that it has an open dense $B$-orbit for a Borel subgroup $B$ of $G$. In his survey of spherical varieties Perrin collects a number of other equivalent characterizations of spherical varieties which we recall here.

\begin{theorem}[{\cite[Theorem 2.1.2]{MR3177371}}] \label{t:perrinSpherical} The normal $G$-variety $X$ is spherical if any of the following hold.
  \begin{enumerate}[label=(\roman*)]
    \item $X$ has an open dense $B$-orbit
    \item $X$ has finitely many $B$-orbits
    \item $\mathbb{C}[X]^{B} = \mathbb{C}$, where $\mathbb{C}[X]$ is the homogeneous coordinate ring of $X$
    \item If $X$ is quasi-projective: For any $\mathcal{L}$ a $G$-linearized line bundle, the $G$-module $\mathrm{H}^0(X,\mathcal{L})$ is multiplicity free.
  \end{enumerate}  
\end{theorem}

\noindent We shall primarily make use of the first and fourth characterizations from Theorem~\ref{t:perrinSpherical}.

\subsection{Algebraic groups}
\label{subsec:AlgGrps}
In this section we will fix notation and briefly cover the algebraic groups background required for this paper. See \cite{MR1102012} for a more detailed treatment.

We will denote the group of invertible $N\times N$ matrices over $\mathbb{C}$ by $\GL_N$. Let $B$ be the standard Borel subgroup of upper triangular matrices with $T$ the standard maximal torus consisting of diagonal matrices. The character group of $T$, $\mathfrak{X}(T) := \mathrm{Hom}_{\mathrm{alg. gp.}}(T, \mathbb{C})$, will be written additively. Any finite dimensional $T$-module $V$ may be written as the sum of weight spaces
\begin{center}
$V = \displaystyle \bigoplus_{\chi \in \mathfrak{X}(T)} V_{\chi}$
\end{center}
where $V_{\chi} := \{ v \in V \mid t v = \chi(t)v, \forall t \in T \}$. We refer to $\chi \in \mathfrak{X}(T)$ as a weight in $V$ if $\mathrm{dim}V_{\chi}\neq 0$. 

For the Adjoint action of $T$ on $\mathfrak{gl}_N:=Lie(\GL_N)$ we define $\Phi$ to be the set of nonzero weights in $\mathfrak{gl}_N$. Then $\Phi = \{ \epsilon_i - \epsilon_j | 1 \leq i,j \leq N \}$ where $\epsilon_i - \epsilon_j$ is the element in $\mathfrak{X}(T)$ that sends the diagonal matrix with $t_1,\ldots,t_n$ on the diagonal to $t_i t_j^{-1}$ in $\mathbb{C}$. The elements in $\Phi$ are referred to as \emph{roots} and $\Phi$ is the \emph{root system} of $\GL_N$ relative to $T$. The Borel subgroup $B$ induces a subset of positive roots $\Phi^{+}= \{ \epsilon_i - \epsilon_j | 1 \leq i < j \leq N \}$ and a subset of simple roots $\Delta= \{ \alpha_i := \epsilon_i - \epsilon_{i+1} | 1 \leq i < N \}$ in $\Phi$.

A \emph{parabolic subgroup} of $\GL_N$ is a closed subgroup containing a Borel subgroup. A \emph{standard parabolic subgroup} is a parabolic subgroup containing $B$. For $1 \leq d < N$ define the \emph{maximal standard parabolic subgroup} $P_d$ as the subgroup containing all elements of $\GL_N$ with a block of zeros of size $N-d \times d$ in the bottom left.
\begin{center}
$P_{d}=\left\{ \left[ \begin{array}{cc}
* & * \\
0_{N-d \times d} & * \\
\end{array} \right] \in \GL_N \right\}$
\end{center}

There is a important bijection between the subsets of the simple roots $\Delta$ and the standard parabolic subgroups. For $I \subseteq \Delta$ we define 
\begin{center}
$P_I = \displaystyle \bigcap_{{\alpha_d} \in \Delta \setminus I } P_d$
\end{center} 

Any standard parabolic subgroup $P_I$ may be written as a semidirect product of its unipotent radical $U_I$ and a reductive subgroup $L_I$ called a Levi factor or Levi subgroup. For $I = \{ \alpha_{i_1},\ldots, \alpha_{i_q} \} \subset \Delta$ with $i_1 \leq \cdots \leq i_q$ and $\Delta \setminus I = \{ \alpha_{j_1},\ldots, \alpha_{j_r} \}$ with $j_1 \leq \cdots \leq j_r$ let $b_{L}=r+1$. Then setting $N_1 = j_1$, $N_k = j_k - j_{k-1}$ for $1 < k < b_L$, and $N_{b_L}=N-j_r$ we have that $N = N_1 + \cdots + N_{b_L}$ and
\begin{center}
$L = \GL_{N_1} \times \cdots \times \GL_{N_{b_L}}$ 
\end{center}

For this reason we will refer to $b_L$ as the \emph{number of blocks of $L$}.

The Weyl group $W$ of $\GL_N$ is generated by the simple reflections $s_{\alpha_i}$ for $\alpha_i \in \Delta$. The group $W$ is isomorphic to the symmetric group of permutations on $N$ letters via the map that identifies $s_{\alpha_i}$ with the transposition $(i, i+1)$. In light of this, we will refer to elements of $W$ by the sequence $(x_1,\cdots,x_N)$ which corresponds to the permutation that sends $i$ to $x_i$. The length of an element $w \in W$, denoted $\ell(w)$, is defined to be the minimum number $k$ such that $w$ may be written as the product of $k$ simple reflections. 

The subsets $I$ of $\Delta$ also index subgroups of $W$. We define $W_I$ to be the subgroup generated by $\{ s_{\alpha_i} | \alpha_i \in I \}$. Then we define a subset of $W$ corresponding to $I$,
\begin{center}
$W^I = \{ w \in W \mid \ell(ww') = \ell(w) + \ell(w'), \textrm{ for all }w' \in W_I \}$.
\end{center}
Then $W = W^I W_I$; that is, any element $w \in W$ can be written as the product $uv$ with $u \in W^I$, $v \in W_I$ and $\ell(w) = \ell(u) + \ell(v)$. Viewed in this way, $W^I$ is the set of minimal length right coset representatives of $W_I$ in $W$.

It will subsequently be convenient to identify the subgroup $W_I$ and subset $W^I$ by their associated parabolic subgroup. In particular, given a standard parabolic subgroup $P=P_I$, we will write $W_P$ and $W^P$ instead of  $W_I$ and $W^I$.

\subsection{Standard monomial theory}
\label{subsec:SMT}
The \emph{Grassmannian} $\Grass_{d,N}$ is the set of all $d$-dimensional subspaces of $\mathbb{C}^{N}$ and can be equipped with a projective variety structure via the Pl\"{u}cker embedding. The Pl\"{u}cker embedding is the map from $\Grass_{d,N}$ to $\mathbb{P}(\bigwedge^d \mathbb{C}^N)$ defined by sending a $d$-dimensional subspace $U$ with basis $\{u_1,\ldots,u_d\}$ to the class $[u_1 \wedge \cdots \wedge u_d]$. This map is well defined (does not depend on choice of basis for $U$) and injective. 

Define $I_{d,N}$ to be the set of strictly increasing positive integer sequences with $d$ values ranging from $1$ to $N$. Explicitly,
\begin{center}
$I_{d,N} = \{ (i_1,\ldots,i_d) | 1 \leq i_1 < \cdots < i_d \leq N \}$.
\end{center}
If $\{ e_1,\ldots,e_N\}$ are the standard basis vectors of $\mathbb{C}^N$, then $\{e_{\tau} := e_{i_1} \wedge \cdots \wedge e_{i_d} \mid \tau = (i_1,\ldots,i_d) \in I_{d,N} \}$ is a basis for $\bigwedge^d \mathbb{C}^N$. Defining $\pl_{\tau} := e_{\tau}^{*}$, we have that $\{\pl_{\tau} \mid \tau \in I_{d,N} \}$ is the dual basis for $(\bigwedge^d \mathbb{C}^N)^{*}$. These $\pl_{\tau}$ are a set of projective coordinates for $\mathbb{P}(\bigwedge^d \mathbb{C}^N)$ called the \emph{Pl\"{u}cker coordinates}. The image of the Grassmannian under the Pl\"{u}cker embedding is cut out scheme theoretically by certain quadratic relations in the Pl\"{u}cker coordinates called the \emph{Pl\"{u}cker relations}. Subsequently we will identify the Grassmannian by its image under the Pl\"{u}cker embedding.

The Grassmannian $\Grass_{d,N}$ is the $\GL_N$-orbit of $[e_{id}:=e_1 \wedge \cdots \wedge e_d]$ for the natural action of $\GL_N$ on $\mathbb{P}(\bigwedge^d \mathbb{C}^N)$. Under this action the $T$-fixed points of $\Grass_{d,N}$ are precisely $[e_{\tau}]$ for $\tau \in I_{d,N}$. The \emph{Schubert variety} $X(\tau)$ is defined to be the Zariski-closure of the $B$-orbit of $[e_{\tau}]$, that is, $X(\tau) := \overline{B[e_{\tau}]}$. The \emph{Bruhat order} on the set $I_{d,N}$ is induced by the containment order on the set of Schubert varieties; $\tau \leq w$ if and only if $X(\tau) \subseteq X(w)$. If $\tau = (i_1,\ldots,i_d)$ and $w=(\ell_1,\ldots,\ell_d)$, then it can be shown that $\tau \leq w$ is equivalent to $i_1 \leq \ell_1,\ldots,i_d \leq \ell_d$. 

\begin{remark}
We noted above that $\Grass_{d,N}$ is the $\GL_N$-orbit of $[e_{id}]$. The isotropy subgroup at $[e_{id}]$ is precisely $P_{d}$. Thus we identify $\Grass_{d,N}$ as the homogeneous space $\GL_N/ P_d$. Consequently, we see that $W^{P_d}$ may be identified with $I_{d,N}$ via the map that sends a $\tau=(i_1,\ldots,i_n) \in W^{P_d}$ to the sequence $(i_1,\ldots,i_d)\uparrow$, where $\uparrow$ indicates that the preceding sequence has been reordered so that it is strictly increasing. In light of these identifications, we shall index the Schubert varieties in $\Grass_{d,N}$ by elements of $W^{P_d}$ while denoting the elements of $W^{P_d}$ by their corresponding sequence in $I_{d,N}$.
\end{remark}

The homogeneous coordinate ring $\mathbb{C}[X(w)]$ of the Schubert variety $X(w)$ induced by the Pl\"{u}cker embedding is a polynomial algebra of the form
\begin{center}
$\mathbb{C}[X(w)] = \mathbb{C}[\pl_{\tau}, \tau \in W^{P_d}] / J$
\end{center}
where $J$ is the homogeneous ideal generated by the Pl\"{u}cker relations and $\{ \pl_{\tau}, \tau \nleq w \}$. We say that a degree r monomial $\pl_{\tau_1}\cdots\pl_{\tau_r} \in \mathbb{C}[X(w)]$ is a \emph{standard monomial on $X(w)$} if $w \geq \tau_1 \geq \cdots \geq \tau_r$. The following theorem illustrates the fundamental importance of standard monomials.
\begin{theorem}[{\cite[Proposition 5.4.7 and Theorem 5.4.8]{MR3408060}}]
\label{T:FundamentalSMT}
The degree r standard monomials on $X(w)$ are a vector space basis of $\mathbb{C}[X(w)]_r$.
\end{theorem}

\subsection{Skew Young diagrams, skew Schur functions, and skew Weyl modules}
\label{subsec:skewYoung}
For a more in depth introduction to the concepts covered in this section see \cite{MR1676282}. A \emph{partition} $\lambda$ is a sequence of positive integers $(\lambda_1,\ldots,\lambda_k)$ such that $\lambda_1 \geq \cdots \geq \lambda_k$. It will be useful to be able to express arbitrarily large partitions. When we write $(a_1^{b_1},\cdots,a_r^{b^r})$ with $a_1 \geq \cdots \geq a_r$ and $b_i \geq 0$ we mean the partition with the first $b_1$ entries equal to $a_1$, the next $b_2$ entries equal to $a_2$, and so on. Note that we will often omit the superscript when it is equal to one. 

We associate to every partition $\lambda$ a \emph{Young diagram}, also denoted $\lambda$, which is a collection of upper left justified boxes with $\lambda_i$ boxes in row $i$. The Young diagram associated to a partition $\lambda$ will be said to have \emph{shape} $\lambda$. For example, the partition $(4,2^2)=(4,2,2)$ corresponds to the Young diagram
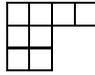
\begin{center}
\ytableausetup{smalltableaux}
\begin{ytableau}
\; & \; & \; & \; \\
\; & \; \\
\; & \; \\
\end{ytableau}
\end{center}
For a partition $\lambda = (\lambda_1,\ldots,\lambda_k)$ we say that the \emph{length} of the partition is the number of entries in the sequence, and denote it by $\ell(\lambda)$. The \emph{size} of the partition is $|\lambda|=\sum \lambda_i$. Given a second partition $\mu = (\mu_1,\ldots,\mu_j)$, we write $\mu \subseteq \lambda$ if $j \leq k$ and $\mu_i \leq \lambda_i$ for $1 \leq i \leq  j$. Equivalently, $\mu \subseteq \lambda$ if the Young diagram of $\mu$ is contained the Young diagram of $\lambda$. 

If $\mu \subseteq \lambda$, we define the skew (Young) diagram $\lambda / \mu$ to be the diagram formed by removing the leftmost $\mu_i$ boxes in row $i$ of $\lambda$ for each row. For example, the skew diagram $(4,2,2) / (2,1)$ is given by
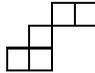
\begin{center}
\ytableausetup{smalltableaux}
\begin{ytableau}
\none & \none & \; & \; \\
\none & \; \\
\; & \; \\
\end{ytableau}
\end{center}

A partition $\lambda = (\lambda_1,\ldots,\lambda_k)$ is said to be a \emph{rectangle} if all entries in the partition are equal to some positive integer $p$. A partition is called a \emph{hook} if there is a positive integer $p$ such that $\lambda_1=p$ and $\lambda_i=1$ for $i\geq 2$. A \emph{fat hook} is a partition such that all entries are equal to either $p$ or $q$ for two positive integers $p$, $q$.

A skew diagram is said to be \emph{basic} if it contains no empty rows or columns. Given a skew diagram $\lambda / \mu$ we define $\tilde{\lambda} / \tilde{\mu}$ to be the basic skew diagram formed by deleting all the empty rows and columns from $\lambda / \mu$. The \emph{$\pi$-rotation} $(\lambda / \mu)^{\pi}$ is the skew diagram that arises by rotating $\lambda / \mu$ through $\pi$ radians. The \emph{conjugate} of a partition $\lambda = (\lambda_1,\ldots,\lambda_k)$ is defined to be the partition $\lambda' = (\lambda'_1,\ldots,\lambda'_{\lambda_1})$ where each $\lambda'_i$ is equal to the number of boxes in column $i$ of $\lambda$. Fix $m$ and $n$ two positive integers. For a partition of length $n$, such that $\lambda \subseteq m^n$, we say that the \emph{$m^n$-complement} of $\lambda$ is $\lambda^\# := (m- \lambda_{n},\ldots,m-\lambda_1)$. Note that $\lambda^\# = (m^n / \lambda)^{\pi}$. The \emph{$m^n$-shortness} of a partition $\lambda$ is the length of the shortest line segment in the path of length $m+n$ from the southwest to northeast corners of $m^n$ that contains the bottom and right contour of $\lambda$.

\begin{example}
Consider the partitions $\lambda = (4,2,2,1)$ and $\mu = (2,2)$. Then
\begin{center}
\ytableausetup{smalltableaux}
$\lambda / \mu =\;\;\;$ \begin{ytableau}
\none & \none & \; & \; \\
\none & \none \\
\; & \; \\
\;  \\
\end{ytableau}
\qquad \qquad 
$\tilde{\lambda} / \tilde{\mu} =\;\;\;$ \begin{ytableau}
\none & \none & \; & \; \\
\; & \; \\
\;  \\
\end{ytableau}
\qquad \qquad 
$(\lambda / \mu)^{\pi} = \;\;\; $ \begin{ytableau}
\none & \none & \none & \; \\
\none & \none & \; & \; \\
\none & \none \\
\; & \; \\
\end{ytableau}
\end{center}
and if $m=n=4$ we have the following Young diagrams. 
\begin{center}
\ytableausetup{smalltableaux}
$\lambda =\;\;\; $\begin{ytableau}
\; & \; & \; & \; \\
\; & \; \\
\; & \; \\
\; \\
\end{ytableau}
\qquad \qquad
$\lambda' =\;\;\; $\begin{ytableau}
\; & \; & \; & \; \\
\; & \; & \; \\
\; \\
\; \\
\end{ytableau}
\qquad \qquad
$\lambda^\# =\;\;\; $\begin{ytableau}
\; & \; & \; \\
\; & \; \\
\; & \; \\
\end{ytableau}
\end{center}
The $m^n$-shortness of $\lambda$ is 1 while the $m^n$-shortness of $\mu$ is 2 as illustrated by the shortest line segments of the  respective paths below.
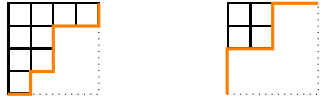
\begin{center}
\begin{tikzpicture}[inner sep=0in,outer sep=0in]
\node (n) {\begin{varwidth}{5cm}{
\begin{ytableau}
\; & \; & \; & \; \\
\; & \; \\
\; & \; \\
\; \\
\end{ytableau}}\end{varwidth}};
\draw[black, dotted] (n.south west)--++(1.21,0)--++(0,1.21);
\draw[very thick,orange] (n.south west)--++(.25*1.21,0)--++(0,.25*1.21)--++(.25*1.21,0)--++(0,.25*1.21)--++(0, .25*1.21)--++(.25*1.21,0)--++(.25*1.21,0)--++(0, .25*1.21);
\end{tikzpicture}
\qquad \qquad
\begin{tikzpicture}[inner sep=0in,outer sep=0in]
\node (n) {\begin{varwidth}{5cm}{
\begin{ytableau}
\; & \; \\
\; & \; \\
\none \\
\none \\
\end{ytableau}}\end{varwidth}};
\draw[black, dotted] (n.south west)--++(1.21,0)--++(0,1.21);
\draw[very thick,orange] (n.south west)--++(0, .25*1.21)--++(0,.25*1.21)--++(.25*1.21,0)--++(.25*1.21,0)--++(0, .25*1.21)--++(0,.25*1.21)--++(.25*1.21,0)--++(.25*1.21,0);
\end{tikzpicture}
\end{center}
\end{example}

Given a skew diagram $\lambda / \mu$, we say that a \emph{filling} of shape $\lambda / \mu$ in $\{ 1,...,n \}$ is an assignment of a value in $\{ 1,...,n \}$ to each box in $\lambda / \mu$. A \emph{tableau} is a filling such that the values in each column increase strictly downwards. A \emph{semistandard tableau} is a tableau such that the values in each row increase weakly along each row. The \emph{weight} of a filling equals $\nu = (\nu_1,...,\nu_n)$ where $\nu_i$ equals the number of boxes with value $i$ in $\ssyt$. An example, from left to right, of a filling, tableau, and semistandard tableau of shape $(4,2,2) / (2,1)$ in $\{ 1,...,4 \}$ is given below. 
\begin{center}
\ytableausetup{smalltableaux}
\begin{ytableau}
\none & \none & 2 & 1 \\
\none & 2 \\
3 & 1 \\
\end{ytableau}
\qquad
\begin{ytableau}
\none & \none & 4 & 3 \\
\none & 2 \\
1 & 3 \\
\end{ytableau}
\qquad
\begin{ytableau}
\none & \none & 2 & 3 \\
\none & 3 \\
2 & 4 \\
\end{ytableau}
\end{center}
The respective weights are $(2,2,1,0)$, $(1,1,2,1)$, and $(0,2,2,1)$.

Given a skew diagram $\lambda / \mu$, the associated \emph{skew Schur function} is
\begin{center}
$\sch{\lambda/\mu} = \displaystyle \sum_{\ssyt} x_1^{\textrm{\# of 1's in }\ssyt}\cdots x_k^{\textrm{\# of k's in }\ssyt}$
\end{center}
where the infinite sum is over all semistandard tableaux of shape $\lambda/\mu$ and $k$ is the maximum value in the semistandard tableau. Then, for a partition $\lambda$, the \emph{Schur function} associated to $\lambda$ is defined to be $\sch{\lambda} := \sch{\lambda/ \emptyset}$ where $\emptyset$ is the zero partition. Though not immediately apparent, the Schur and skew Schur functions are symmetric. In fact, the ring of symmetric functions has a basis given by the Schur functions. The Littlewood-Richardson coefficients $\lr{\mu}{\nu}{\lambda}$ appear as the structure coefficients for multiplication in this ring. That is, for partitions $\mu$ and $\nu$ we have
\begin{equation}
\sch{\mu}\sch{\nu} = \displaystyle \sum_{\lambda} \lr{\mu}{\nu}{\lambda} \sch{\lambda}
\end{equation}
where the sum is over all partitions $\lambda$ such that $|\lambda| = |\mu| + |\nu|$. The Littlewood-Richardson coefficients $\lr{\mu}{\nu}{\lambda}$ also appear in the expansion of the skew Schur functions
\begin{equation}
\sch{\lambda / \mu} = \displaystyle \sum_{\nu} \lr{\mu}{\nu}{\lambda} \sch{\nu},
\end{equation}
where the sum is over all $\nu$ such that $|\lambda| - |\mu| = |\nu|$. A skew Schur function is \emph{multiplicity-free} if, in the expansion of the skew Schur function into the basis of Schur functions, all the nonzero Littlewood-Richardson coefficients are equal to 1. 

Fix a positive integer $N$. The \emph{skew Schur polynomial} $\sch{\lambda/\mu}(x_1,\ldots,x_N)$ is a specialization of the skew Schur function $\sch{\lambda/\mu}$ achieved by setting $x_m=0$ for all $m>N$. The \emph{Schur polynomial} $\sch{\lambda}(x_1,\ldots,x_N)$ is defined to be $\sch{\lambda/ \emptyset}(x_1,\ldots,x_N)$. The Schur polynomials associated to partitions of length less than or equal $N$ give a basis for the ring of symmetric functions in variables $x_1,\ldots,x_N$. Thus
\begin{equation}
\label{e:SchurPolynomial}
\sch{\lambda / \mu}(x_1,\ldots,x_N) = \displaystyle \sum_{\nu} \lr{\mu}{\nu}{\lambda} \sch{\nu}(x_1,\ldots,x_N),
\end{equation}
where the sum is over all $\nu$ with $\ell(\nu) \leq N$ such that $|\lambda| - |\mu| = |\nu|$.

\begin{remark}
\label{r:suffMultFree}
We will say that a skew Schur polynomial is multiplicity free if the expansion into the basis of Schur polynomials is multiplicity free. Importantly, if the skew Schur function $\sch{\lambda/\mu}$ is multiplicity free then the skew Schur polynomial $\sch{\lambda/\mu}(x_1,\ldots,x_N)$ is multiplicity free. However, the converse is not true, as there might be partitions $\nu$ of length greater than $N$ such that $\lr{\mu}{\nu}{\lambda} > 1$ but none with length less than or equal $N$.

\end{remark}

The first of the two identities below may be found in \cite{MR1676282} while the second follows trivially from the Littlewood-Richardson rule (see Section~\ref{subsec:LRRule}).
\begin{equation}
\label{e:SchurPiRotation}
\sch{\lambda / \mu} = \sch{(\lambda / \mu)^{\pi}}
\end{equation}
\begin{equation}
\label{e:SchurBasic}
\sch{\lambda / \mu} = \sch{\tilde{\lambda} / \tilde{\mu}}
\end{equation}

The second identity implies that the classification of multiplicity-free Schur functions reduces to a classification of basic multiplicity-free Schur functions. This classification was achieved by Thomas and Yong in \cite{MR2583223} (see also \cite[Theorem 4.3]{MR2737323}).

\begin{theorem}[\cite{MR2583223}]
\label{T:skewMultFree}
The basic skew Schur function $s_{\lambda / \mu}$ is multiplicity-free if and only if $\lambda$ and $\mu$ satisfy one or more of the following conditions:
\begin{enumerate}
\item $\mu$ or $\lambda^{\#}$ is the zero partition
\item $\mu$ or $\lambda^{\#}$ is a rectangle of $m^{n}$-shortness 1
\item $\mu$ is a rectangle of $m^{n}$-shortness 2 and $\lambda^{\#}$ is a fat hook (or vice versa)
\item $\mu$ is a rectangle and $\lambda^{\#}$ is a fat hook of $m^{n}$-shortness 1 (or vice versa)
\item $\mu$ and $\lambda^{\#}$ are rectangles
\end{enumerate}
where $m=\lambda_1$, $n=\lambda'_1$, and $\lambda^{\#}$ is the $m^n$-complement of $\lambda$.
\end{theorem}

The skew diagrams also index certain distinguished representations of $\GL_N$. For $\lambda / \mu$ a skew diagram we denote the corresponding \emph{skew Weyl module}, equivalently Schur functor, by $\mathbb{W}^{\lambda / \mu}(\mathbb{C}^N)$ (see \cite[\S 6.1]{MR1153249} for the details of this construction). We will normally simplify this notation by writing $\mathbb{W}^{\lambda / \mu}$ as long as no confusion will arise from doing so.

For a partition $\lambda$ the corresponding \emph{Weyl module} is $\mathbb{W}^{\lambda} := \mathbb{W}^{\lambda / \emptyset}$. The Weyl modules $\mathbb{W}^{\lambda}$ such that $\ell(\lambda) \leq N$ are precisely the polynomial irreducible representations of $\GL_N$. We have that $\GL_N$ is completely reducible since it is a reductive group and we are working over $\mathbb{C}$. Thus any $\GL_N$-representation may by written uniquely, up to isomorphism, as a direct sum of irreducible representations. The decomposition of $\mathbb{W}^{\lambda / \mu}$ into irreducible representations has a particularly nice description in terms of the Littlewood-Richardson coefficients.
\begin{equation}
\label{e:SkewWeylDecomp}
\mathbb{W}^{\lambda / \mu} = \displaystyle \bigoplus (W^{\nu})^{\oplus \lr{\mu}{\nu}{\lambda}}
\end{equation}
\begin{remark}
\label{r:weylModuleMultFree}
This identity follows from the fact that $\sch{\lambda / \mu}(x_1,\ldots,x_N)$ is the character of $\mathbb{W}^{\lambda / \mu}$. In particular, this implies that $\mathbb{W}^{\lambda / \mu}$ has a multiplicity-free decomposition into irreducible $\GL_N$-modules if and only if $\sch{\lambda / \mu}(x_1,\ldots,x_N)$ is multiplicity-free. Following Remark~\ref{r:suffMultFree} we conclude that Theorem~\ref{T:skewMultFree} gives sufficient conditions on $\lambda / \mu$ for $\mathbb{W}^{\lambda / \mu}$ to have a multiplicity free decomposition.
\end{remark}

\subsection{Computing Littlewood-Richardson coefficients via the Littlewood-Richardson rule}
\label{subsec:LRRule}
Many of our results will rely on the ability to compute certain Littlewood-Richardson coefficients. To facilitate these computations we recall an identity and the Littlewood-Richardson Rule\cite[Section~5]{MR1464693}. The identity is a non-trivial symmetry of the Littlewood-Richardson coefficients and its proof may be found in \cite{MR1676282}, it states that
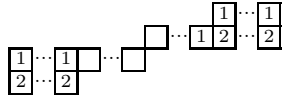
\begin{equation}
\label{equation:littlewoodRichardsonIdentities}
c_{\mu, \nu}^{\lambda} = c_{\nu, \mu}^{\lambda}.
\end{equation}

The \emph{row word} of a semistandard tableau $\ssyt$, denoted $w_{row}(\ssyt)$, is the values in $\ssyt$ written from left to right and bottom to top. If a row word equals $t_1,...,t_r$ we say that it is a \emph{reverse lattice word} if the number $i$ appears at least as often as $i+1$ in every reversed subsequence $t_r,t_{r-1},...,t_{s+1},t_s$. A semistandard tableau $\ssyt$ such that $w_{row}(\ssyt)$ is a reverse lattice word is a \emph{semistandard Littlewood-Richardson tableau}. 

\begin{example} Consider the two semistandard tableaux below.
\begin{center}
\qquad \qquad \qquad
\begin{ytableau}
\none & \none & 1 & 1 \\
\none & 1 \\
1 & 3 \\
2 \\
3 \\
\end{ytableau}
\qquad \qquad \qquad
\begin{ytableau}
\none & \none & 1 & 1 \\
\none & 2 \\
1 & 3 \\
2 \\
3 \\
\end{ytableau}
\end{center}
Their respective row words are 3,2,1,3,1,1,1 and 3,2,1,3,2,1,1. The second is a reverse lattice word while the first is not; the number of 3's in 1,1,1,3 is greater than the number of 2's. Thus only the second is a semistandard Littlewood-Richardson tableau.
\end{example}

\begin{proposition}[{\cite[Proposition 5.3]{MR1464693} Littlewood-Richardson rule}] \label{p:LittlewoodRichardson} The Littlewood-Richardson coefficient $c_{\mu,\nu}^{\lambda}$ is equal to the number of semistandard Littlewood-Richardson tableaux of shape $\lambda / \mu$ and weight $\nu$.
\end{proposition}

\noindent We will now use this proposition to prove two lemma involving the Littlewood-Richardson coefficients that will be useful in Section \ref{sec:class}.

\begin{lemma}
\label{l:MultFreePolyNotFunction}
Let $\lambda=(r^{N},p,q)$, $\mu=(a,b)$ be two partitions with $0 < b \leq a < r$ and $0 < q \leq p < r$. Then the skew Schur polynomial $\sch{\lambda / \mu}(x_1,\ldots,x_N)$ is multiplicity free and hence $\mathbb{W}^{\lambda / \mu}$ has a multiplicity free decomposition into irreducible $\GL_N$-modules.
\end{lemma}
\begin{proof}
We begin by noting that Theorem~\ref{T:skewMultFree} implies that $\sch{\lambda / \mu}$ is not multiplicity free, nonetheless the result still holds.  Let $\nu=(\nu_1,\ldots,\nu_m)$ be a partition of length $m \leq N$ such that $|\lambda| - |\mu| = |\nu|$; we will show that $\lr{\mu}{\nu}{\lambda} \leq 1$. To do this we will once again use  \eqref{equation:littlewoodRichardsonIdentities} to equivalently show that $\lr{\nu}{\mu}{\lambda} \leq 1$. We begin by considering how we might fill the skew diagram $(r^{N},p,q) / \nu$ with $a$ ones and $b$ twos such that a semistandard tableau results. The fact that we can only use ones and twos immediately means that we can only successfully construct such a semistandard tableau when $(r^{N},p,q) / \nu$ has no more than two boxes in any column. Since the length of $\nu$ is less than or equal $N$ this restricts us to those $\nu$ such that the associated basic form of $(r^{N},p,q) / \nu$ is
\begin{center}
\begin{ytableau}
\none & \none &  \none & \none & \none  & \none  & \none & \none & \none & & \none[\smdots] &  \\
\none & \none &  \none & \none & \none  & \none  & & \none[\smdots] & & & \none[\smdots] &  \\
 & \none[\smdots] &  & & \none[\smdots] &   \\
 & \none[\smdots] &  \\
\end{ytableau}
\end{center}
We will show that there are no choices when filling such a skew Young diagram with $a$ ones and $b$ twos if we wish the row word to be a reverse lattice word. First, we are forced to fill all the columns with two boxes with ones and twos. If we wish for the row word to be a reverse lattice word we are then forced to put a one in the rightmost column of row two that contains a single box. Note that this is true even if there are no columns with two boxes on the right.
\begin{center}
\begin{ytableau}
\none & \none &  \none & \none & \none  & \none  & \none & \none & \none & 1 & \none[\smdots] & 1 \\
\none & \none &  \none & \none & \none  & \none  & & \none[\smdots] & 1 & 2 & \none[\smdots] & 2 \\
 1& \none[\smdots] &1  & & \none[\smdots] &   \\
 2& \none[\smdots] &2  \\
\end{ytableau}
\end{center}
Now we are forced to put ones in the rest of the boxes in row two since we need a semistandard tableau. Finally, there is only one way to fill in the remaining empty boxes in row three such that the result is a semistandard tableau. It is possible that at some point in the preceding discussion we could not proceed because we would have had to use more than $a$ ones or $b$ twos; in this case the associated Littlewood-Richardson coefficient is zero. Otherwise, we could fill in the skew Young diagram but every choice was prescribed. Thus $\sch{\lambda / \mu}(x_1,\ldots,x_N)$ is multiplicity free since any Littlewood-Richardson coefficient in \eqref{e:SchurPolynomial} is equal to 0 or 1.
\end{proof}

\begin{lemma} Let $n \geq 0$ and $m > 0$.
\label{lemma:LRSphericalClassH}
\begin{enumerate}
\item \label{lemma:LRSphericalClassH1} If $\lambda=(2^n,1,1)$, $\mu=(1)$, and $\nu=(2^n,1)$, then $c_{\mu,\nu}^{\lambda}=1$.
\item \label{lemma:LRSphericalClassH2} If $\lambda=(2^{n+1})$, $\mu=(1)$, and $\nu=(2^n,1)$, then $c_{\mu,\nu}^{\lambda}=1$.
\item \label{lemma:LRSphericalClassH3} If $\lambda=(2^{m},1)$, $\mu=(1,1)$, and $\nu=(2^{m-1},1)$, then $c_{\mu,\nu}^{\lambda}=1$.
\item \label{lemma:LRSphericalClassH4} If $\lambda=(3^m,2,1)$, $\mu=(2,1)$ and $\nu=(3^{m-1},2,1)$, then $c_{\mu,\nu}^{\lambda}=2$.
\item \label{lemma:LRSphericalClassH5} Let $\lambda$ and $\mu$ be partitions such that $\mu \subset \lambda$. Let $\nu = (n)$ where $n = |\lambda|-|\mu|$. Then $\lr{\mu}{\nu}{\lambda}=1$. 
\end{enumerate}
\end{lemma}
\begin{proof}
Note that in these partitions, when the exponent of an entry is 0 we simply omit that entry from the partition. For example, when $n=0$ in \eqref{lemma:LRSphericalClassH1} we have $\lambda=(2^0,1,1)=(1,1)$.

\noindent \eqref{lemma:LRSphericalClassH1}: The identity \eqref{equation:littlewoodRichardsonIdentities} implies that $c_{\mu,\nu}^{\lambda}=c_{\nu,\mu}^{\lambda}$. Using Proposition \ref{p:LittlewoodRichardson}, we find $c_{\nu,\mu}^{\lambda}$ by counting the number of semistandard Littlewood-Richardson tableaux of shape $(2^n,1,1)/(2^n,1)$ with weight $(1)$. Since $(2^n,1,1)/(2^n,1)$ is a single box there is exactly one semistandard Littlewood-Richardson tableau with weight $\mu=(1)$. Hence $c_{\nu,\mu}^{\lambda}=1$.

\noindent \eqref{lemma:LRSphericalClassH2}: As $\lambda / \nu=(2^{n+1})/(2^n,1)$ is a single box, Proposition \ref{p:LittlewoodRichardson} and \eqref{equation:littlewoodRichardsonIdentities} imply that $c_{\mu,\nu}^{\lambda}=c_{\nu,\mu}^{\lambda}=1$.

\noindent \eqref{lemma:LRSphericalClassH3}: Once again we use \eqref{equation:littlewoodRichardsonIdentities} to see that $c_{\mu,\nu}^{\lambda}=c_{\nu,\mu}^{\lambda}$. Then $\lambda / \nu=(2^{m},1)/(2^{m-1},1)$ is two disconnected boxes. There are two possible fillings of these disconnected boxes with weight $(1,1)$. Only the filling with a 1 in the upper right box and a 2 in the lower left box is a semistandard Littlewood-Richardson tableau. Thus we have $c_{\mu,\nu}^{\lambda}=c_{\nu,\mu}^{\lambda}=1$.

\noindent \eqref{lemma:LRSphericalClassH4}: As in the previous cases we calculate $c_{\nu,\mu}^{\lambda}$. The skew diagram $\lambda / \nu=(3^m,2,1)/(3^{m-1},2,1)$ is three disconnected boxes. There are two possible fillings of these boxes with weight $(2,1)$ whose row word is a reverse lattice word. Thus we have $c_{\mu,\nu}^{\lambda}=c_{\nu,\mu}^{\lambda}=2$.

\noindent \eqref{lemma:LRSphericalClassH5}: It is clear that any filling of $\lambda / \mu$ with $n = |\lambda|-|\mu|$ ones can be done in exactly one way. Further, the row word with all ones is a reverse lattice word. Hence $\lr{\mu}{\nu}{\lambda}=1$.
\end{proof}

\section{The decomposition of the homogeneous coordinate ring}
\label{sec:Decomp}

Fix positive integers $d < N$. Then $\Grass_{d,N} = \GL_N / P_d$. Let $w \in W^{P_d}$ and let $P$ be a standard parabolic subgroup that acts on the Schubert variety $X(w)$ by left multiplication. This induces an action of the Levi part of $P$, which we will denote by $L$, on $X(w)$, which in turn induces an action of $L$ on the homogeneous coordinate ring $\mathbb{C}[X(w)]$. In \cite{HodgesLakshmibai}, the authors give a combinatorial description of the decomposition of $\mathbb{C}[X(w)]$ into irreducible $L$-modules for this induced action. We recall this result as well as the relevant definitions and notation.

As remarked in Section~\ref{subsec:AlgGrps}, the standard parabolic subgroup $P$ must be of the form $P_I$ for some $I=\{ \alpha_{i_1},\ldots,\alpha_{i_q} \} \subseteq \Delta$ with $i_1 < \cdots < i_q$ and $\Delta \setminus I = \{ \alpha_{j_1},\ldots, \alpha_{j_r} \}$ with $j_1 \leq \cdots \leq j_r$. Recall that $b_L := r+1$. If $N_1 = j_1$, $N_k = j_k - j_{k-1}$ for $1\leq k < b_L$, and $N_{b_L}=N-j_r$ then $N = N_1 + \cdots + N_{b_L}$ and $L = \GL_{N_1} \times \cdots \times \GL_{N_{b_L}}$. Let $\Bl_{L,k}=\{ j_{k-1}+1,\ldots,j_k \}$ for $1 \leq k \leq b_L$ where $j_0 = 0$ and $j_{b_L}=N$. Then the subsets $\Bl_{L,1}$,\ldots,$\Bl_{L,b_L}$, which we refer to as the \emph{blocks of L}, are a partition of $\{ 1,\ldots,N \}$. It is an easy check that $N_k = | \Bl_{L, k} |$. We give an example of these subsets below in Example~\ref{e:blocks}.

Denote by $H_w$ the subset of $W^{P_d}$ containing all $\tau \leq w$. Then for a $\theta \in H_w$ we say that $\theta$ is a \emph{degree 1 head of type $L$} if $X_{\theta}$ is a $L$-stable Schubert subvariety of $X_w$.
\begin{proposition}[{\cite[Proposition~3.1.6]{HodgesLakshmibai}}]
\label{p:HeadLComb}
A $\theta \in H_w$ is a degree 1 head of type $L$ if and only if $\theta \cap \Bl_{L,k}$ is maximal for all $1 \leq k \leq b_L$; explicitly we require that for all $m \in \theta \cap \Bl_{L,k}$ and $n \in \Bl_{L,k} \setminus \theta \cap \Bl_{L,k}$ we have $m > n$. 
\end{proposition}
We will denote the subset of $H_w$ that contains all the degree 1 heads of type $L$ by $\He_{L,1}$. Fix a positive integer $r$. A \emph{degree r head of type $L$} is a sequence $\underline{\theta} = (\theta_1,\ldots,\theta_r)$ such that $\theta_i \in \He_{L,1}$. A degree r head is \emph{standard} if in addition $\theta_1 \geq \cdots \geq \theta_r$. Define 
\begin{center}
$\He_{L,r} = \{ (\theta_1,\ldots,\theta_r) | \theta_i \in \He_{L,1} \}$
\end{center}
and
\begin{center}
$\He_{L,r}^{std} = \{ (\theta_1,\ldots,\theta_r) \in \He_{L,r} | \theta_1 \geq \cdots \geq \theta_r \}$.
\end{center}

One final set of definitions is required before we can describe the decomposition from \cite{HodgesLakshmibai}. Given a standard degree r head of type L we associate it to a collection of $k$ skew diagrams. Let $\underline{\theta} = (\theta_1,\ldots,\theta_r) \in \He_{L,r}^{std}$. We begin by defining the semistandard tableau $\tab_{\underline{\theta}}$ of shape $(r^d)$ by letting the columns of $\tab_{\underline{\theta}}$ correspond to the $\theta_i$ in reverse order. Explicitly, the values from top to bottom in column $c$ of $\tab_{\underline{\theta}}$ correspond to the first to last entries in $\theta_{r-c+1}$ for $1 \leq c \leq r$.

Fix a $k$ such that $1 \leq k \leq b_L$. Then $\tab_{\underline{\theta}}^{(k)}$ is the basic semistandard tableau formed by first deleting all boxes in $\tab_{\underline{\theta}}$ with values not in $\Bl_{L,k}$ and then deleting all empty rows and columns (we omit the step of subtracting $j_k$ from the value in each box as is done in \cite{HodgesLakshmibai} since here we only care about the shape of the skew semistandard tableaux). This semistandard tableau has some shape, which we will write as $\lambda_{\underline{\theta}}^{(k)} / \mu_{\underline{\theta}}^{(k)}$. Finally, we define the $L$-module associated to $\underline{\theta}$ by
\begin{center}
$\mathbb{W}_{\underline{\theta}} := \mathbb{W}^{\lambda_{\underline{\theta}}^{(1)} / \mu_{\underline{\theta}}^{(1)}}(\mathbb{C}^{N_1}) \otimes \cdots \otimes \mathbb{W}^{\lambda_{\underline{\theta}}^{(b_L)} / \mu_{\underline{\theta}}^{(b_L)}}(\mathbb{C}^{N_{b_L}})$.
\end{center} 

\begin{theorem}[{\cite[Theorem~3.5.4]{HodgesLakshmibai}}]
\label{t:LeviSchubertMainDecomp}
For a fixed $r$, we have a decomposition of $\mathbb{C}[X(w)]_r$ into $L$-modules given by
\begin{center}
$\mathbb{C}[X(w)] \cong \displaystyle \bigoplus_{\underline{\theta} \in \He_{L,r}^{std}} \mathbb{W}^*_{\underline{\theta}}$
\end{center}
where $\mathbb{W}^*_{\underline{\theta}}$ is the $L$-module dual of $\mathbb{W}_{\underline{\theta}}$.
\end{theorem}

As $\mathbb{W}_{\underline{\theta}}$ is a tensor product of skew Weyl modules, the decomposition of $\mathbb{C}[X(w)]_r$ into irreducible $L$-modules may then be achieved via \eqref{e:SkewWeylDecomp}.

\begin{example}
\label{e:blocks}
Set $d=3$ and $N=9$ and consider $w=(2,7,9) \in W^{P_d}$. The Schubert variety $X(w)$ is $L=L_I$-stable for $I=\{\alpha_1,\alpha_3,\alpha_4,\alpha_5,\alpha_6,\alpha_8 \}$. Then $\Delta \setminus I = \{ \alpha_2, \alpha_7 \}$ and we have that $L=\GL_2 \times \GL_5 \times \GL_2$. Then $b_L=3$ and the blocks of $L$ are
\begin{center}
$\Bl_{L, 1} = \{ 1, 2 \} \qquad \Bl_{L, 2} = \{ 3, 4, 5, 6, 7 \} \qquad \Bl_{L, 3} = \{ 8, 9 \}.$
\end{center}
The degree one heads of type $L$ are $(1, 2, 7)$, $(2,6,7)$, and $(2, 7, 9)$. One standard degree three head is $\underline{\theta} = ((2,7,9) , (2,6,7) , (1,2,7))$. We will now construct the skew semistandard tableaux and skew diagrams associated to $\underline{\theta}$. We have
\begin{center}
$\tab_{\underline{\theta}} = \begin{ytableau}
1 & 2 & 2 \\
2 & 6 & 7 \\
7 & 7 & 9 \\
\end{ytableau}$.
\end{center}
We then create a basic skew semistandard tableau $\tab_{\underline{\theta}}^{(k)}$ for each $1 \leq k \leq b_L$ by removing boxes not in $\Bl_{L,k}$ and deleting empty rows and columns.
\begin{center}
$\tab_{\underline{\theta}}^{(1)} = \begin{ytableau}
1 & 2 & 2 \\
2 \\
\end{ytableau} \qquad \qquad \tab_{\underline{\theta}}^{(2)} = \begin{ytableau}
\none & 6 & 7 \\
7 & 7 \\
\end{ytableau} \qquad \qquad \tab_{\underline{\theta}}^{(3)} = \begin{ytableau}
9 \\
\end{ytableau}$
\end{center}
\end{example}
The associated skew diagrams are $\lambda_{\underline{\theta}}^{(1)} / \mu_{\underline{\theta}}^{(1)} = (3,1)/\emptyset$, $\lambda_{\underline{\theta}}^{(2)} / \mu_{\underline{\theta}}^{(2)} = (3,2)/(1)$, and $\lambda_{\underline{\theta}}^{(3)} / \mu_{\underline{\theta}}^{(3)}=(1)/\emptyset$. This implies that the $L$-module associated to the degree 3 head $\underline{\theta}$ is
\begin{center}
$\mathbb{W}_{\underline{\theta}} := \mathbb{W}^{(3,1)/\emptyset}(\mathbb{C}^{2}) \otimes \mathbb{W}^{(3,2)/(1)}(\mathbb{C}^{5}) \otimes \mathbb{W}^{(1)/\emptyset}(\mathbb{C}^{2})$
\end{center}
\section{Reductions and Multiplicity Criterion}
\label{sec:multFreeClass}

It will be easier to state our classification result if we first perform some reductions. We define a \emph{Levi-Schubert quadruple} to be the datum $(w, d, N, L)$ where $d < N$ are positive integers, $w=(\ell_1,\cdots,\ell_d) \in W^{P_d}$, and $L$ is a Levi subgroup of $\GL_N$. A Levi-Schubert quadruple is \emph{stable} if $X(w)$ is $L$-stable for the action of $L$ by left multiplication. A stable Levi-Schubert quadruple is \emph{multiplicity free} if the decomposition of $\mathbb{C}[X(w)]$ into irreducible $L$-modules is multiplicity free. A stable Levi-Schubert quadruple is \emph{spherical} if $X(w)$ is a spherical $L$-variety. We say $(w, d, N, L)$ is \emph{reduced} if $\ell_1 \neq 1$ and $\ell_d=N$. If $(w, d, N, L)$ is not reduced we define its \emph{reduction} $(\overline{w}, \overline{d}, \overline{N}, \overline{L})$ as follows. The fact that the quadruple is not reduced implies that $w = (1,\ldots,p,\ell_{p+1},\ldots,\ell_d)$ for some $p \geq 0$ or $\ell_d\neq N$ with $\ell_{p+1} \neq p+1$. 

Note that throughout the paper we will assume that $w$ is not the identity; in the case when $w$ is the identity, we have that $X(w)$ is a point space and hence if it is $L$-stable for some Levi subgroup $L$, then it is trivially spherical. Set 
\begin{center}
$\overline{w} = (\ell_{p+1}-p,\ldots,\ell_d-p)$ \\
$\overline{d} = d-p$ \\
$\overline{N} = \ell_d - p$.
\end{center}
Finally, we define $\overline{L}$ to be the image of $L$ under the map $\mathrm{pr}_w$ which is the composition of a diagonal projection map and a projection map
\begin{center}
$\mathrm{pr}_w:\GL_N \longrightarrow \GL_p \times \GL_{\ell_d - p} \times \GL_{N - \ell_d} \longrightarrow \GL_{\ell_d - p}$.
\end{center}

\begin{lemma}
\label{l:reductioniso}
If $(\overline{w}, \overline{d}, \overline{N}, \overline{L})$ is the reduction of the Levi-Schubert quadruple $(w, d, N, L)$, then $X(w) \cong X(\overline{w})$ as varieties and $\mathbb{C}[X(w)] \cong \mathbb{C}[X(\overline{w})]$ as graded $\mathbb{C}$-algebras.
\end{lemma}
\begin{proof}
Consider the commutative diagram below.

\[
\begin{tikzcd}
X(\overline{w}) \arrow{r}{} \arrow[hook]{d}{} & X(w) \arrow{r}{} \arrow[hook]{d}{} & X(w) \arrow[hook]{d}{} \\
\Grass_{\overline{d}, \overline{N}} \arrow[hook]{r}{} \arrow[hook]{d}{A} & \Grass_{d, \overline{N}} \arrow[hook]{r}{} \arrow[hook]{d}{B} & \Grass_{d, N} \arrow[hook]{d}{C}\\
\mathbb{P}(\bigwedge^{\overline{d}}\mathbb{C}^{\overline{N}}) \arrow[hook]{r}{} & \mathbb{P}(\bigwedge^{d}\mathbb{C}^{\overline{N}}) \arrow[hook]{r}{} & \mathbb{P}(\bigwedge^{d}\mathbb{C}^{N})
\end{tikzcd}
\]

Here $A$, $B$, and $C$ are the Pl\"{u}cker embeddings discussed in Section \ref{subsec:SMT} and we wish to show that the top arrows are isomorphisms. Recall that we defined the set $H_w := \{ \tau \in W^{P_d} | \tau \leq w \}$; every element in $H_w$ is of the form $(1,...,p,t_{p+1},...,t_{d})$ with $t_{d}\leq \ell_{d}$. We next define 

\begin{equation}
\label{e:HasseIso}
\begin{array}{c}
\iota:H_w \longrightarrow H_{\overline{w}} \\
(1,...,p,t_{p+1},...,t_{d}) \longmapsto (t_{p+1}-p,...,t_{d}-p)
\end{array}
\end{equation}

Note that this map is a bijection with $\iota(w)=\overline{w}$, and for $\tau$, $\gamma \in H_w$ we have $\tau \leq \gamma$ if and only if $\iota(\tau) \leq \iota(\gamma)$. Thus the poset $(H_w, \leq)$ is isomorphic to the poset $(H_{\overline{w}},\leq)$. It is well known that the Schubert variety $X(w)$ is cut out scheme theoretically from $\Grass_{d,N}$ by the equations $\{\pl_{\tau}=0\mid\tau \nleq w\}$ and similarly $X(\overline{w})$ from $\Grass_{\overline{d}, \overline{N}}$ by $\{\pl_{\overline{\tau}}=0\mid \overline{\tau} \nleq \overline{w}\}$ (see for example~\cite[Chapter 5]{MR3408060}). Hence the above isomorphism of posets implies $\mathbb{C}[X(w)] \cong \mathbb{C}[X(\overline{w})]$ as $\mathbb{C}$-algebras, which further implies that $X(w) \cong X(\overline{w})$ as varieties.
\end{proof}

\begin{lemma}
\label{l:reductionstable}
If the Levi-Schubert quadruple $(w, d, N, L)$ is stable then its reduction $(\overline{w}, \overline{d}, \overline{N}, \overline{L})$ is stable. Further, if $w = (1,\ldots,p,\ell_{p+1},\ldots,\ell_d)$ then
\begin{center}
$L = \GL_{a_1} \times \cdots \times \GL_{a_r} \times \overline{L} \times \GL_{c_1} \times \cdots \times \GL_{c_t}$
\end{center}
for some positive integers $a_1 + \cdots + a_r = p$ and $c_1 + \cdots + c_r = N - \ell_d$.
\end{lemma}
\begin{proof}
As discussed above, if $(w, d, N, L)$ is not reduced then $w = (1,\ldots,p,\ell_{p+1},\ldots,\ell_d)$ for some $p \geq 0$ or $\ell_d\neq N$ with $\ell_{p+1} \neq p+1$. We will prove the result for $p>{}0$ and $\ell_d\neq N$ as the cases where $p=0$ or $\ell_d= N$ are simpler versions of this general case.
Let $Q_w$ and $Q_{\overline{w}}$ be the largest standard parabolic subgroups that act on $X(w)$ and $X(\overline{w})$ respectively. Then $Q_w = P_{I_w}$ for some $I_w \subseteq \Delta$. If $\Delta \setminus I_w = \{ \alpha_{j_1},\ldots,\alpha_{j_m} \}$ with $j_1 < \cdots < j_m$, then by \cite[Proposition~3.1.1]{HodgesLakshmibai}, we have
\begin{equation}
\label{e:formLw}
\Delta \setminus I_w = \{ \alpha_b \mid \exists m\text{ with } b=\ell_m\text{ and }\ell_m + 1\neq \ell_{m+1} \}.
\end{equation}
This implies that $\alpha_p$ and $\alpha_{\ell_d}$ are elements of $\Delta \setminus I_w$, in particular, $\Delta \setminus I_w = \{ \alpha_{p},\alpha_{j_2},\ldots,\alpha_{j_{m-1}},\alpha_{\ell_d} \}$. Thus, using \eqref{e:HasseIso} and \eqref{e:formLw}, we have $Q_{\overline{w}}=P_{I_{\overline{w}}}$ with $\Delta \setminus I_{\overline{w}} = \{ \alpha_{{j_2} - p},\ldots,\alpha_{{j_{m-1}}-p} \}$.

Let $Q=P_I$ and $\overline{Q}=P_{\overline{I}}$ be the parabolic subgroups with Levi parts $L$ and $\overline{L}$ respectively. Since $X(w)$ is $L$-stable it is $Q$-stable and hence $\Delta \setminus I_w \subseteq \Delta \setminus I$. We will show that $\Delta \setminus I_{\overline{w}} \subseteq \Delta \setminus \overline{I}$ which will imply that $X(\overline{w})$ is $\overline{L}$-stable. The form of $\mathrm{pr}_w$ implies that 
\begin{equation}
\label{e:formIw}
\Delta \setminus \overline{I}=\{ \alpha_{b-p} | \alpha_b \in I\;\mathrm{and}\;p<b<N-\ell_d \}.
\end{equation}
Now if we take a $\alpha_{j_b-p} \in \Delta \setminus I_{\overline{w}}$, we have that $\alpha_{j_b} \in \Delta \setminus I_w \subset \Delta \setminus I$. Thus \eqref{e:formIw} implies that $\alpha_{j_b-p} \in \Delta \setminus \overline{I}$.

The fact that $\Delta \setminus I_w \subseteq \Delta \setminus I$ implies that $\alpha_p$ and $\alpha_{\ell_d}$ are elements of $\Delta \setminus I$. Thus
\begin{center}
$L = \GL_{a_1} \times \cdots \times \GL_{a_r} \times \GL_{b_1} \times \cdots \times \GL_{b_s} \times \GL_{c_1} \times \cdots \times \GL_{c_t}$
\end{center}
for some positive integers with $a_1 + \cdots + a_r = p$ and $c_1 + \cdots + c_r = N - \ell_d$. Then $\overline{L}:=\mathrm{pr}_w(L)=\GL_{b_1} \times \cdots \GL_{b_s}$, which completes the proof.
\end{proof}

\begin{proposition}
\label{l:reductionSpherical}
Let $(w, d, N, L)$ be a stable Levi-Schubert quadruple. Then $(w, d, N, L)$ is spherical if and only if its reduction $(\overline{w}, \overline{d}, \overline{N}, \overline{L})$ is spherical.
\end{proposition}
\begin{proof}
An element in $\Grass_{\overline{d},\overline{N}}$ is of the form $\overline{g} P_{\overline{d}}$ for some $\overline{g} \in \GL_{\overline{N}}$. The injective map $i: \Grass_{\overline{d},\overline{N}} \hookrightarrow \Grass_{d,N}$ from Lemma \ref{l:reductioniso} takes the element $\overline{g} P_{\overline{d}}$ to $g P_{d}$ where $g \in \GL_N$ is a block diagonal matrix of the form
\begin{center}
$\begin{bmatrix}
\;\mathrm{I}_p & & \\
& \overline{g} & \\
& & \mathrm{I}_{N-\ell_d}
\end{bmatrix}$
\end{center}
and $\mathrm{I}_n$ denotes the identity matrix of size $n \times n$. Define the action of an element $\jmath \in L$ on $\Grass_{\overline{d},\overline{N}}$ by left multiplication by $\mathrm{pr}_w(\jmath)$. We claim that $i$ is $L$-equivariant for this action. Recalling the form of $L$ from Lemma~\ref{l:reductionstable} an element $\jmath \in L$ is a block diagonal of the form
\begin{center}
$\begin{bmatrix}
\; \jmath_1  & & \\
& \jmath_2 & \\
& & \jmath_3
\end{bmatrix}$
\end{center}
where $\jmath_1 \in \GL_{a_1} \times \cdots \times \GL_{a_r}$, $\jmath_2 \in \overline{L}$, and $\jmath_3 \in \GL_{c_1} \times \cdots \times \GL_{c_t}$. Thus our claim is that 
\begin{center}
$i(\mathrm{pr}_w(\jmath)\, \overline{g})(\modp{P_d})=\begin{bmatrix}
\;\mathrm{I}_p & & \\
& \jmath_2 \overline{g} & \\
& & \mathrm{I}_{N-\ell_d}
\end{bmatrix}(\modp{P_d})=\begin{bmatrix}
\;\jmath_1 & & \\
& \jmath_2 \overline{g} & \\
& & \jmath_3
\end{bmatrix}\modp{P_d}=\jmath\, i(\overline{g})(\modp{P_d})$
\end{center}
Note that $\ell_d \geq d$ implies $N - \ell_d \leq N-d$. This, combined with the fact that $p < d$, implies the block diagonal
\begin{center}
$a = \begin{bmatrix}
\; \jmath_1^{-1}  & & \\
& I_{\overline{N}} & \\
& & \jmath_3^{-1}
\end{bmatrix}$
\end{center}
is an element of $P_d$. Thus $i(\mathrm{pr}_w(\jmath)\, \overline{g}) = \jmath\, i(\overline{g}) a$, which implies our claim and hence $i$ is $L$-equivariant. In Lemma~\ref{l:reductioniso} we showed that $X(\overline{w}) \cong X(w)$ under this map. Since  $X(\overline{w})$ is $\overline{L}$-stable it is $L$-stable for the action defined above, while $X(w)$ is $L$-stable by hypothesis. 

Thus we have an isomorphism of $L$-varieties, indicating that $X(\overline{w})$ will be a spherical $L$-variety if and only if $X(w)$ is a spherical $L$-variety. We conclude since $X(\overline{w})$ will be a spherical $L$-variety if and only if it is a spherical $\overline{L}$-variety.
\end{proof}

\begin{proposition}
\label{p:multFreeSphericalEquiv}
The stable Levi-Schubert quadruple $(w, d, N, L)$ is spherical if and only if it is multiplicity free.
\end{proposition}
\begin{proof}
We start by proving that multiplicity free implies spherical (see \cite[Proposition 4.0.1]{HodgesLakshmibai} for an alternative proof). The Pl\"{u}cker embedding of $\Grass_{d,N}$ into $\mathbb{P}(\bigwedge^d\mathbb{C}^N)$ was given in Section~\ref{subsec:SMT}. Let $\mathfrak{L}$ be the corresponding very ample line bundle on $\Grass_{d,N}$ for this embedding. Let $\tilde{\mathfrak{L}}$ be an $L$-linearized line bundle on $X(w)$. Every line bundle on $X(w)$ is the restriction of a line bundle on $\Grass_{d,N}$ and any such line bundle is $\GL_N$-linearized and of the form $\mathfrak{L}^{\otimes r}$ for some integer $r$. Hence $H^0(X(w),\tilde{\mathfrak{L}}) = H^0(X(w),\mathfrak{L}^{\otimes r}|_{X(w)})$. When $r$ is strictly less than zero we have that $H^0(X(w),\mathfrak{L}^{\otimes r}|_{X(w)})=0$ by \cite[Theorem 5.6.4]{MR3408060}. When $r$ is non-negative $H^0(X(w),\mathfrak{L}^{\otimes r}|_{X(w)})$ is the degree $r$ portion of the homogeneous coordinate ring of $X(w)$, which by hypothesis is a multiplicity free $L$-module. In both cases $H^0(X(w),\tilde{\mathfrak{L}}) = H^0(X(w),\mathfrak{L}^{\otimes r}|_{X(w)})$ is a multiplicity free $L$-module and hence by Theorem~\ref{t:perrinSpherical} $X(w)$ is a spherical $L$-variety.

 For the other direction, suppose that $(w, d, N, L)$ is spherical. The homogeneous coordinate ring $\mathbb{C}[X(w)]=\bigoplus_{r\geq 0}H^0(X(w),\mathfrak{L}^{\otimes r}|_{X(w)})$. Via Theorem~\ref{t:perrinSpherical}, we know that if $X(w)$ is a spherical $L$-variety and $\beta$ is an $L$-linearized line bundle on $X(w)$, then $H^0(X(w),\beta)$ is a multiplicity free $L$-module. As $\mathfrak{L}^{\otimes r}|_{X(w)}$ is the restriction of a $\GL_N$-linearized line bundle on $\Grass_{d,n}$ to an $L$-stable subvariety, $\mathfrak{L}^{\otimes r}|_{X(w)}$ is $L$-linearized. The three preceding statements combine to imply that each individual degree piece of $\mathbb{C}[X(w)]$ is multiplicity free. In \cite[Theorem 4.1.2]{HodgesLakshmibai}, it is shown, via the explicit description of the decomposition of $\mathbb{C}[X(w)]$ into irreducible $L$-modules, that an irreducible $L$-submodule in a fixed degree of $\mathbb{C}[X(w)]$ can not be isomorphic to an irreducible $L$-submodule in a different degree. Hence $(w, d, N, L)$ is multiplicity free.
\end{proof}
\begin{corollary}
\label{c:mainSphericalReduction}
Let $(w, d, N, L)$ be a stable Levi-Schubert quadruple. Then $(w, d, N, L)$ is multiplicity free if and only if its reduction $(\overline{w}, \overline{d}, \overline{N}, \overline{L})$ is multiplicity free.
\end{corollary}

Towards the completion of our classification of spherical Levi-Schubert quadruples we give a multiplicity free criterion derived from Theorem~\ref{t:LeviSchubertMainDecomp}.

\begin{proposition}
\label{p:MultFreeHighestLevel}
The stable Levi-Schubert quadruple $(w, d, N, L)$ is multiplicity free if and only if the following two properties are satisfied for all $r \geq 1$.
\begin{center}
\begin{tabular}{cp{5.25in}}
($\MC$) & For all degree r standard heads $\underline{\theta} \in \He_{L, r}^{std}$, the $\GL_{N_k}$ skew Weyl module $\mathbb{W}^{\lambda_{\underline{\theta}}^{(k)} / \mu_{\underline{\theta}}^{(k)}}(\mathbb{C}^{N_{k}})$ is multiplicity free for all $1 \leq k \leq b_L$. \\
($\MCC$) & Let $\underline{\theta},\underline{\theta}' \in \He_{L, r}^{std}$ be two degree r standard heads such that $\underline{\theta} \neq \underline{\theta}'$. If $M$ is an irreducible $L$-submodule of $\mathbb{W}_{\underline{\theta}}$ and $M'$ is an irreducible $L$-submodule of $\mathbb{W}_{\underline{\theta}'}$, then $M \ncong M'$.
\end{tabular}
\end{center}
\end{proposition}
\begin{proof}
This is immediate by Theorem~\ref{t:LeviSchubertMainDecomp} and the aforementioned fact from \cite[Theorem 4.1.2]{HodgesLakshmibai} that there can be no isomorphisms between irreducible $L$-submodules in different degrees of the homogeneous coordinate ring.
\end{proof}

\section{Classification}
\label{sec:class}
Fix a stable, reduced Levi-Schubert quadruple $(w,d,N,L)$ with $w=(\ell_1,\ldots,\ell_d)$. If $Q=P_I$ is the parabolic subgroup with Levi part $L$, then $I = \{ \alpha_{i_1},\ldots,\alpha_{i_m} \} \subseteq \Delta$ for some $i_1 < \cdots < i_q$ and $\Delta \setminus I = \{ \alpha_{j_1},\ldots, \alpha_{j_r} \}$ for some $j_1 \leq \cdots \leq j_r$. Let $N_k$ for $1 \leq k \leq b_L$ be as in Section \ref{sec:Decomp}. The goal of this section is to give explicit criterion for when such a quadruple is multiplicity free (equivalently spherical). First we will need some additional notation.

Define the non-negative integers $h_1,\ldots,h_{b_L}$ by $h_k = |\left\{j | \ell_j \in \Bl_{L, k} \right\}|$. Then for all $1 \leq k \leq b_L$ we have
\begin{equation}
\label{e:hForm}
0 \leq h_k \leq N_k.
\end{equation}

Additionally, since the subsets $\Bl_{L,k}$ for $1 \leq k \leq b_L$ partition the set $\{ 1,\ldots,N \}$, each entry in $w$ is in some block. This means
\begin{equation}
\label{e:hSumd}
d = h_1 + \cdots + h_{b_L}.
\end{equation}
Since $(w,d,N,L)$ is reduced, the fact that $\ell_1 \neq 1$ and $\ell_d = N$ imply
\begin{equation}
\label{e:hReduced}
\begin{array}{c}
h_1 < N_1 \\[5pt]
h_{b_L} \geq 1
\end{array}
\end{equation}
\begin{remark}
\label{r:MaxLevi}
If $L_{max}$ is the maximal Levi acting on $X(w)$ we can refine the bounds in \eqref{e:hForm} slightly. In this case, by \eqref{e:formLw}, we see that each $j_b$, such that $\alpha_{j_b} \in \Delta \setminus I$, is equal to some entry in $w$. Suppose that $h_k = N_k$. Then $j_{k-1}+1,\ldots,j_k$ are all entries in $w$, and this would imply, by \eqref{e:formLw}, that $\alpha_{j_{k-1}} \notin \Delta \setminus I$. If $k=1$ this contradicts the fact that $(w,d,N,L_{max})$ is reduced with $\ell_1 \neq 1$. Otherwise, for $k > 1$, this is a contradiction of the definition of $\Bl_{L_{max},k-1}$. Thus $h_k < N_k$. Further, we know that $j_k \in \Bl_{L_{max},k}$, and since $j_k$ is an entry in $w$ this means $h_k > 0$. One additional important fact that follows from \eqref{e:formLw} is that $N_k > 1$. Summarizing, if $L_{max}$ is the maximal Levi acting on $X(w)$, then
\begin{equation}
\label{e:hFormMax}
\begin{array}{c}
0 < h_k < N_k \\[5pt]
N_k > 1
\end{array}
\end{equation}
for all $1 \leq k \leq b_{L_{max}}$.
\end{remark}

The other notation we will need is an alternative method for indexing the degree 1 heads of type $L$. For non-negative integers $m_1,\ldots,m_{b_L}$ define the sequence
\begin{center}
$\Theta(m_1,\ldots,m_{b_L}) := (j_1 - m_1 + 1,\ldots,j_1,j_2-m_2+1,\ldots,j_2,\ldots,j_{b_L}-m_{b_L}+1,\ldots,j_{b_L})$
\end{center}
where $j_{b_L}=j_{r+1}$ is defined to be equal to $N$. Here our convention is that when $m_k$ is zero we omit the corresponding subsequence. Thus, such a sequence will always be of length $m_1 +\cdots + m_{b_L}$. In general this will not even be an element of $W^{P_d}$, however, if certain properties hold it will be a degree 1 head of type $L$.
\bibliographystyle{alpha}

\begin{lemma}
\label{l:headCriterion}
Let $m_1,\ldots,m_{b_L}$ be non-negative integers. Then $\Theta(m_1,\ldots,m_{b_L})$ is a degree 1 head of type L if and only if the following three criterion are satisfied.
\begin{enumerate}[label=(\arabic*)]
\item $m_1+\cdots+m_{b_L}=d$
\item $m_k \leq N_k$ for all $1 \leq k \leq b_L$
\item $m_1 + \cdots + m_k \geq h_1 + \cdots + h_k$ for all $1 \leq k \leq b_L$
\end{enumerate}
Further, the degree 1 heads of type L are in bijection with the non-negative integers $m_1,\ldots,m_{b_L}$ satisfying these conditions.
\end{lemma}
\begin{proof}
The sequence $\Theta(m_1,\ldots,m_{b_L})$ has d entries if and only if $m_1+\cdots+m_{b_L}=d$. Since $N_k = j_k - j_{k-1}$, the sequence $\Theta(m_1,\ldots,m_{b_L})$ will have no repeated values if and only if $m_k \leq N_k$ for all $1 \leq k \leq b_L$. Thus the first two conditions are satisfied if and only if $\Theta(m_1,\ldots,m_{b_L})$ is an element of $W^{P_d}$. Identifying $w$ with $\Theta(h_1,\ldots,h_{b_L})$, it is not difficult to check that for $\Theta(m_1,\ldots,m_{b_L}) \in W^{P_d}$, $\Theta(m_1,\ldots,m_{b_L}) \leq \Theta(h_1,\ldots,h_{b_L})$ if and only if condition 3 is satisfied. Thus $\Theta(m_1,\ldots,m_{b_L}) \in H_w$ if and only if conditions 1, 2, and 3 are satisfied.

All that remains is to verify that the three criterion imply that $\Theta(m_1,\ldots,m_{b_L})$ satisfies the combinatorial description of degree 1 heads of type $L$ given in Proposition~\ref{p:HeadLComb}. As $\Bl_{L,k} = \{j_{k-1}+1,...,j_k \}$, this is immediate by the definition of $\Theta(m_1,\ldots,m_{b_L})$. The fact that any head of type L can be written as $\Theta(m_1,\ldots,m_{b_L})$ for some $m_1,\ldots,m_{b_L}$ satisfying these conditions also follows trivially from Proposition~\ref{p:HeadLComb}.
\end{proof}
\begin{corollary}
\label{c:headCriterionMain} Let $\underline{\Theta}$ be a degree r head and fix a $k$ such that $1\leq k \leq b_L$.
\begin{enumerate}
\item \label{c:headCriterionMain1} Boxes in $\tab_{\underline{\theta}}$ with values in $\Bl_{L,k}$ can appear only in row $h_1 + \cdots + h_{k-1} + 1$ and below.
\item \label{c:headCriterionMain2} Boxes in $\tab_{\underline{\theta}}$ with values less than those in $\Bl_{L,k}$ can appear only in row $N_1 + \cdots + N_{k-1}$ and above.
\item \label{c:headCriterionMain3} Suppose that $h_{k+1} + \cdots + h_{b_L} < p$. Then boxes in $\tab_{\underline{\theta}}$ with values greater than those in $\Bl_{L,k}$ can appear only in the bottom $p-1$ rows.
\end{enumerate}
\end{corollary}

\begin{example}
As in Example \ref{e:blocks}, we let $d=3$, $N=9$ and consider the Schubert variety $X(w)$ where  $w=(2,7,9) \in W^{P_d}$. Then $X(w)$ is $L=L_I$-stable for $I=\{\alpha_1,\alpha_3,\alpha_4,\alpha_5,\alpha_6,\alpha_8 \}$ and $\Delta \setminus I = \{ \alpha_2, \alpha_7 \}$. Then $b_L=3$, $j_1 = 2$, $j_2 = 7$, $j_3 = 9$, and
\begin{center}
$\Bl_{L, 1} = \{ 1, 2 \} \qquad \Bl_{L, 2} = \{ 3, 4, 5, 6, 7 \} \qquad \Bl_{L, 3} = \{ 8, 9 \}.$
\end{center}
Then 
\begin{center}
$\begin{array}{r@{\hspace{4pt}}l}
\Theta(3,2,3) &= (0,1,2,6,7,7,8,9) \\[4pt]

\Theta(0,1,2) &= (7,8,9) \\[4pt]

\Theta(2,1,0) &= (1,2,7) \\[4pt]

\Theta(1,1,1) &= (2,7,9) \\

\end{array}$
\end{center}
and we have that $\Theta(0,1,2), \Theta(2,1,0), \Theta(1,1,1) \in  W^{P_d}$. Note that of these three only $\Theta(2,1,0)$ and $\Theta(1,1,1)$ are degree 1 heads of type $L$ since $\Theta(0,1,2) \nleq w$.

This indexing method is particularly useful for studying the skew Young diagrams associated to a degree $r$ head; we will primarily use it to exhibit degree r heads with specific properties. Consider the degree 3 head $\underline{\theta} = (\Theta(1,1,1) , \Theta(2,1,0), \Theta(2,1,0))$. Summing the first entry of each head in $\underline{\theta}$ we see that the skew semistandard tableau $\tab_{\underline{\theta}}$ will have 5 boxes with values in $\Bl_{L, 1}$ and so $\lambda_{\underline{\theta}}^{(1)} / \mu_{\underline{\theta}}^{(1)}$ will have 5 boxes. It is not difficult to check that $\lambda_{\underline{\theta}}^{(1)} / \mu_{\underline{\theta}}^{(1)} = (3,2)/\emptyset$. The skew diagrams associated to the other blocks may be worked out in this way as well.
\end{example}

\begin{remark}
\label{r:headFormat}
The non-negative integers $h_1,\ldots,h_{b_L}$ and their relation to $N_1,\ldots,N_{b_L}$ give a lot of information about possible degree 1 heads of type $L$ and hence about possible degree $r$ heads. When considering a degree r head $\underline{\theta}$ and its associated semistandard tableau $\tab_{\underline{\theta}}$ we may say the following. Suppose that $h_1 + \cdots + h_{k-1} + 1 \geq N_1 + \cdots + N_{k-1}$; then \eqref{e:hForm} and \eqref{e:hReduced} imply that $h_1 = N_1 + 1, h_2=N_2,\ldots,h_{k-1}=N_{k-1}$. Then Corollary \ref{c:headCriterionMain}\eqref{c:headCriterionMain1}\eqref{c:headCriterionMain2} implies that in $\tab_{\underline{\theta}}$ the boxes containing values in $\Bl_{L,k}$ can only appear in row $h_1 + \cdots + h_{k-1} + 1$ and greater, while boxes with values less than those in $\Bl_{L,k}$ can appear in rows no greater than $N_1 + \cdots + N_{k-1} = h_1 + \cdots + h_{k-1} + 1$. These combine to imply, since $\tab_{\underline{\theta}}$ is semistandard, that in the skew diagram $\lambda_{\underline{\theta}}^{(k)} / \mu_{\underline{\theta}}^{(k)}$ defined in Section \ref{sec:Decomp} we must have $\mu_{\underline{\theta}}^{(k)}$ equal to either $\emptyset$ or $(p)$ for some positive integer $p$.

\end{remark}

\begin{proposition}
\label{p:MCCSatCrit}
Let $(w,d,N,L)$ be a reduced Levi-Schubert quadruple. Then the multiplicity-free criterion $\MCC$ from Proposition~\ref{p:MultFreeHighestLevel} is satisfied if and only if for all $1 < k < b_L - 1$ at least one the two following conditions holds.
\begin{enumerate}
\item $h_1 + \cdots + h_k + 1 \geq N_1 + \cdots + N_k$
\item $h_{k+1} + \cdots + h_{b_L} < 2$
\end{enumerate}
\end{proposition}
\begin{proof}
\noindent $(\Leftarrow)$ Let $\underline{\theta}$ and $\underline{\theta}'$ be two standard degree $r$ heads. Further, let $M \subset \mathbb{W}_{\underline{\theta}}$ and $M' \subset \mathbb{W}_{\underline{\theta}'}$ be irreducible $L$-submodules. Then $M \cong \mathbb{W}^{\nu_{\underline{\theta}}^{(1)}} \otimes \cdots \otimes \mathbb{W}^{\nu_{\underline{\theta}}^{(b_L)}}$ and $M' \cong \mathbb{W}^{\nu_{\underline{\theta}'}^{(1)}} \otimes \cdots \otimes \mathbb{W}^{\nu_{\underline{\theta}'}^{(b_L)}}$ for some partitions $\nu_{\underline{\theta}}^{(k)}$ and $\nu_{\underline{\theta}'}^{(k)}$. Now suppose that $M \cong M'$, this implies that $|\nu_{\underline{\theta}}^{(k)}|=|\nu_{\underline{\theta}'}^{(k)}|$ for all $1 \leq k \leq b_L$. This implies that the number of boxes in $\tab_{\underline{\theta}}$ with values in $\Bl_{L,k}$ is equal to the number of boxes in $\tab_{\underline{\theta}'}$ with values in $\Bl_{L,k}$ for all blocks.

We can say more in the cases where $k=1$ or $k=b_L$. It is not difficult to see from the definition of $\tab_{\underline{\theta}}^{(1)}$ and $\tab_{\underline{\theta}'}^{(1)}$ that $\mu_{\underline{\theta}}^{(1)}=\emptyset$ and $\mu_{\underline{\theta}'}^{(1)}=\emptyset$. Thus $\mathbb{W}^{\lambda_{\underline{\theta}}^{(1)} / \mu_{\underline{\theta}}^{(1)}} \cong \mathbb{W}^{\lambda_{\underline{\theta}}^{(1)}}$ is irreducible, as is $\mathbb{W}^{\lambda_{\underline{\theta}'}^{(1)} / \mu_{\underline{\theta}'}^{(1)}} \cong \mathbb{W}^{\lambda_{\underline{\theta}'}^{(1)}}$. Hence $M \cong M'$ implies that the partition $\lambda_{\underline{\theta}}^{(1)}= \lambda_{\underline{\theta}'}^{(1)}$. In particular, this implies that the boxes in $\tab_{\underline{\theta}}$ in $\Bl_{L,1}$ are exactly the same as the boxes in $\tab_{\underline{\theta}'}$ in $\Bl_{L,1}$. Similarly, for $k=b_L$, $\lambda_{\underline{\theta}}^{(b_L)}=(p^q)$ for some $p\leq r$ and $q \leq d$. In addition, $\lambda_{\underline{\theta}'}^{(b_L)}=(a^b)$ for some $a\leq r$ and $b \leq d$. Thus $\mathbb{W}^{\lambda_{\underline{\theta}}^{(b_L)} / \mu_{\underline{\theta}}^{(b_L)}}$ and $\mathbb{W}^{\lambda_{\underline{\theta}'}^{(b_L)} / \mu_{\underline{\theta}'}^{(b_L)}}$ are irreducible. Once again, $M \cong M'$ will then imply that the boxes in $\tab_{\underline{\theta}}$ with values in $\Bl_{L,b_L}$ are exactly the same as the boxes in $\tab_{\underline{\theta}'}$ with values in $\Bl_{L,b_L}$.

So far we have not used our hypothesis; the above arguments always hold. We will now show that the hypothesis implies that the boxes in $\tab_{\underline{\theta}}$ in $\Bl_{L,k}$ are exactly the same as the boxes in $\tab_{\underline{\theta}'}$ in $\Bl_{L,k}$ for all blocks. This will imply our desired result, since a degree r head is completely determined by the block membership of its entries (this is an easy exercise using Proposition~\ref{p:HeadLComb}, or see \cite[Lemma~3.1.10]{HodgesLakshmibai}).

The hypothesis implies one of two possible cases. The first case is that there exists a minimal $n$ such that for $1 \leq k < n$ we have $h_1 + \cdots + h_k + 1 \geq N_1 + \cdots + N_k$ and for $n \leq k < b_L - 1$ we have $h_{k+1} + \cdots + h_{b_L} < 2$. This implies, by Corollary \ref{c:headCriterionMain}\eqref{c:headCriterionMain3}, that all boxes in $\tab_{\underline{\theta}}$ and $\tab_{\underline{\theta}'}$ with values greater than those in $\Bl_{L,n}$ must appear in the last row. Additionally, we know that $\tab_{\underline{\theta}}$ and $\tab_{\underline{\theta}'}$ are semistandard and the number of boxes with values in each block is equal. These combine to imply that the boxes in $\tab_{\underline{\theta}}$ in $\Bl_{L,k+1}$ are exactly the same as the boxes in $\tab_{\underline{\theta}'}$ in $\Bl_{L,k+1}$ for $n \leq k < b_L$. Now consider the boxes with values in $\Bl_{L,n}$. As we noted in Remark~\ref{r:headFormat}, $h_1 + \cdots + h_{n-1} + 1 \geq N_1 + \cdots + N_{n-1}$ implies that any boxes in $\tab_{\underline{\theta}}$ in rows greater than $h_1 + \cdots + h_{n-1} + 1$ can not have values less than those in $\Bl_{L,n}$. Thus any boxes in these rows that do not have values larger than those in $\Bl_{L,n}$ must be filled by values in $\Bl_{L,n}$. By Corollary \ref{c:headCriterionMain}\eqref{c:headCriterionMain1} the remaining boxes in $\tab_{\underline{\theta}}$ with values in $\Bl_{L,n}$ must all be in row $h_1 + \cdots + h_{n-1} + 1$ directly to the left of the values larger than those in $\Bl_{L,n}$. The same argument holds for the location of boxes in $\tab_{\underline{\theta}'}$ with values in $\Bl_{L,n}$. Thus, since they each have the same number of boxes with values in each fixed block, the boxes in $\tab_{\underline{\theta}}$ in $\Bl_{L,n}$ are exactly the same as the boxes in $\tab_{\underline{\theta}'}$ in $\Bl_{L,n}$. Now proceed inductively all the way down to $\Bl_{L,2}$ to conclude the argument.

The second case is that there exists no minimal $n$ as in the first case. This simply means that $h_1 + \cdots + h_k + 1 \geq N_1 + \cdots + N_k$ for $1 \leq k < b_L - 1$. In this case we proceed immediately to the inductive step in the first case.

\noindent $(\Rightarrow)$ For this direction we will prove the contrapositive, that is, our hypothesis will be that there exists a $k$ such that $h_1 + \cdots + h_k + 1 < N_1 + \cdots + N_k$ and $h_{k+1} + \cdots + h_{b_L} \geq 2$. Our goal will be to show that $\MCC$ is never satisfied by exhibiting two nonequal standard degree r heads with isomorphic irreducible $L$-submodules.

Let $b$ be the maximum index less than or equal $k$ such that $h_b < N_b$, and $c$ the minimum index greater than $k$ such that $h_c > 0$. Both these indices must exist by our hypothesis. Now we will define the non-negative integers $m_1 ,\ldots , m_{b_L}$ as follows. If $b=k$, then set $m_i = h_i$ for all $1 \leq i \leq k$. Otherwise, set $m_b = h_b + 1$, $m_k = h_k - 1$, and $m_i = h_i$ for all $1 \leq i \leq k$ with $i \neq b$ and $i \neq k$. If $c=k+1$, then set $m_i = h_i$ for all $k+1 \leq i \leq b_L$. Otherwise, set $m_{k+1} = h_{k+1} + 1$, $m_c = h_c - 1$, and $m_i = h_i$ for all $k+1 \leq i \leq b_L$ with $i \neq k+1$ and $i \neq c$. 

Then $m_1 ,\ldots , m_{b_L}$ are non-negative integers satisfying the conditions of Lemma~\ref{l:headCriterion}. We also have that $m_k < N_k$ and $m_{k+1} > 0$. If $b\neq1$, then $m_1 = h_1 < N_1$ by \eqref{e:hReduced}. If $b=1$, then this would imply by the maximality of $b$ that $h_2 = N_2,\ldots,h_k=N_k$ and so by the hypothesis $m_1 = h_1 + 1 < N_1$. Thus, in both cases, $m_1 < N_1$. If $c \neq b_L$ then $m_{b_L}=h_{b_L} > 0$ by \eqref{e:hReduced}. Otherwise, if $c=b_L$, the minimality of $c$ implies that $h_{k+1}=\cdots=h_{b_L-1}=0$, which by the hypothesis implies that $h_{b_L} \geq 2$. Thus $m_{b_L} = h_{b_L}-1 \geq 1$. Hence in both cases, $m_{b_L} > 0$. Using these properties we may construct four degree 1 heads of type $L$.
\begin{center}
$\begin{array}{r@{\hspace{4pt}}l}
\theta_1=& \Theta(m_1 ,\ldots , m_{b_L}) \\[4pt]
\theta_2=& \Theta(m_1+1,m_2,\ldots,m_{k-1},m_k + 1, m_{k+1} - 1,m_{k+2},\ldots, m_{b_L - 1},m_{b_L} - 1) \\[4pt]
\theta_3=& \Theta(m_1,\ldots,m_{k-1},m_k + 1, m_{k+1} - 1,m_{k+2},\ldots,m_{b_L}) \\[4pt]
\theta_4=& \Theta(m_1+1,m_2,\ldots,m_{b_L - 1},m_{b_L} - 1) \\
\end{array}$
\end{center}

It is an easy check to verify that all four of these satisfy the conditions from Lemma~\ref{l:headCriterion}. Recalling that for two degree 1 heads, $\Theta(p_1 ,\ldots , p_{b_L}) \leq \Theta(q_1 ,\ldots , q_{b_L})$ if and only if $p_1 + \cdots + p_k \geq q_1 + \cdots + q_k$ for all $1 \leq k \leq b_L$ we see that $\theta_2 \leq \theta_1$ and $\theta_4 \leq \theta_3$. Thus $\underline{\theta} = (\theta_1, \theta_2)$ and $\underline{\theta}' = (\theta_3, \theta_4)$ are two non-equal standard degree 2 heads. We have
\begin{center}
$\mathbb{W}_{\underline{\theta}} = \mathbb{W}^{(2^{m_1},1)} \otimes \mathbb{W}^{(2^{m_2},1)/(1)} \otimes \cdots \otimes \mathbb{W}^{(2^{m_k},1,1)/(1)} \otimes \mathbb{W}^{(2^{m_{k+1}},1)/(1,1)} \otimes \mathbb{W}^{(2^{m_{k+2}},1)/(1)} \otimes \cdots \otimes \mathbb{W}^{(2^{m_{b_L}})/(1)}$
\end{center}
and
\begin{center}
$\mathbb{W}_{\underline{\theta}'} = \mathbb{W}^{(2^{m_1},1)} \otimes \mathbb{W}^{(2^{m_2},1)/(1)} \otimes \cdots \otimes \mathbb{W}^{(2^{m_k + 1})/(1)} \otimes \mathbb{W}^{(2^{m_{k+1}-1},1)} \otimes \mathbb{W}^{(2^{m_{k+2}},1)/(1)} \otimes \cdots \otimes \mathbb{W}^{(2^{m_{b_L}})/(1)}$
\end{center}

The two $L$-modules listed above only differ in the $k$th and $(k+1)$th factors. The Weyl module $\mathbb{W}^{(2^{m_{k}},1)}$ is a $\GL_{N_k}$ submodule of both $\mathbb{W}^{(2^{m_k},1,1)/(1)}$ and $\mathbb{W}^{(2^{m_k + 1})/(1)}$ by Lemma~\ref{lemma:LRSphericalClassH}\eqref{lemma:LRSphericalClassH1}\eqref{lemma:LRSphericalClassH2}. Additionally, Lemma~\ref{lemma:LRSphericalClassH}\eqref{lemma:LRSphericalClassH3} implies that $\mathbb{W}^{(2^{m_{k+1}-1},1)}$ is a $\GL_{N_{k+1}}$-submodule of $\mathbb{W}^{(2^{m_{k+1}},1)/(1,1)}$. These combine to imply that
\begin{center}
$\mathbb{W}^{(2^{m_1},1)} \otimes \mathbb{W}^{(2^{m_2},1)/(1)} \otimes \cdots \otimes \mathbb{W}^{(2^{m_{k}},1)} \otimes \mathbb{W}^{(2^{m_{k+1}-1},1)} \otimes \mathbb{W}^{(2^{m_{k+2}},1)/(1)} \otimes \cdots \otimes \mathbb{W}^{(2^{m_{b_L}})/(1)}$
\end{center}
is an $L$-submodule of both $\mathbb{W}_{\underline{\theta}}$ and $\mathbb{W}_{\underline{\theta}'}$. This indicates that criterion $\MCC$ is violated and concludes our proof.
\end{proof}

\begin{proposition}
\label{p:MCSatCrit}
Let $(w,d,N,L)$ be a reduced Levi-Schubert quadruple. The multiplicity-free criterion $\MC$ from Proposition~\ref{p:MultFreeHighestLevel} is satisfied if and only if for all $1 < k < b_L$ at least one the five following conditions holds
\begin{enumerate}
\item $N_k = 1$
\item $h_1 + \cdots + h_{k-1} + 1 \geq N_1 + \cdots + N_{k-1}$
\item $h_k=N_k$ with $h_1 + \cdots + h_{k-1} + 2 \geq N_1 + \cdots + N_{k-1}$
\item $h_k > 0$ with $h_{k+1} + \cdots + h_{b_L} < 2$
\item $h_k=0$ with $h_{k+1} + \cdots + h_{b_L} \leq 2$
\end{enumerate}
\end{proposition}
\begin{proof}
\noindent $(\Leftarrow)$ Let $\underline{\theta}$ be a standard degree r head. Fix a $1 < k < b_L$, and set $\lambda = \lambda_{\underline{\theta}}^{(k)}$ and $\mu = \mu_{\underline{\theta}}^{(k)}$. We will show that $\mathbb{W}^{\lambda / \mu}$ is multiplicity free, and thus $\MC$ is satisfied. We have five possible cases.

\noindent \textbf{Case 1}: $N_k = 1$. Using Lemma~\ref{lemma:LRSphericalClassH}\eqref{lemma:LRSphericalClassH5} we see that any Littlewood-Richardson coefficient with $\nu$ of length 1 is either zero or one. Thus the decomposition of $\mathbb{W}^{\lambda / \mu}$ into irreducible $\GL_1$-modules is always multiplicity free.

\noindent \textbf{Case 2}: $h_1 + \cdots + h_{k-1} + 1 \geq N_1 + \cdots + N_{k-1}$. As noted in Remark~\ref{r:headFormat}, this implies that in the skew diagram $\mathbb{W}^{\lambda / \mu}$, we must have $\mu=(p)$ or $\mu=\emptyset$. In either case, we have by Theorem~\ref{T:skewMultFree} and Remark~\ref{r:weylModuleMultFree} that $\mathbb{W}^{\lambda / \mu}$ is multiplicity free.

\noindent \textbf{Case 3}: $h_k=N_k$ with $h_1 + \cdots + h_{k-1} + 2 \geq N_1 + \cdots + N_{k-1}$. Note that this also implies that $h_1 + \cdots + h_{k} + 2 \geq N_1 + \cdots + N_{k}$. Using similar reasoning as in Remark~\ref{r:headFormat} we see that the earliest row of $\tab_{\underline{\theta}}$ that can contain values in $\Bl_{L,k}$ is row $h_1 + \cdots h_{k-1} +  1$ and the latest row is $h_1 + \cdots + h_k + 2$. We further know that the earliest row in which values greater than $\Bl_{L,k}$ can appear is row $h_1 + \cdots + h_k + 1$. Additionally, the latest row in which values less than $\Bl_{L,k}$ can appear is row $N_1 + \cdots + N_{k-1} \leq h_1 + \cdots + h_{k-1} + 2$. These combine to imply that $\lambda / \mu$ is of the form  $(r^{s},p,q) / (a,b)$ for some $0 \leq q \leq p < r$, $0 \leq b \leq a < r$, and $h_k-2 \leq s \leq h_k=N_k$. The cases where either $b$ or $q$ are zero result in $\mu=(a)$ or $(\lambda)^\# = (r-p)$ respectively. In these cases Theorem~\ref{T:skewMultFree} and Remark~\ref{r:weylModuleMultFree} give us our multiplicity free result. When both $b$ and $q$ are not zero we must have that $s=h_k=N_k$. By Lemma~\ref{l:MultFreePolyNotFunction} we have that the skew Schur polynomial $\sch{\lambda / \mu}(x_1,\ldots,x_{N_k})$ is multiplicity free. Thus $\mathbb{W}^{\lambda / \mu}$ is multiplicity free. 

\noindent \textbf{Case 4}: $h_k>0$ with $h_{k+1} + \cdots + h_{b_L} < 2$. In this case, Corollary \ref{c:headCriterionMain}\eqref{c:headCriterionMain3} implies that the boxes with values greater than those in $\Bl_{L,k}$ can only appear in row $d$ of $\tab_{\underline{\theta}}$. Setting $m$ equal to the first entry in $\lambda$ and $n$ equal to the number of entries in $\lambda$, the preceding remarks imply that the $m^n$-complement $(\lambda)^\# = (p)$ or  $(\lambda)^\# = \emptyset$. In either case, Theorem~\ref{T:skewMultFree} and Remark~\ref{r:weylModuleMultFree} indicate that $\mathbb{W}^{\lambda / \mu}$ is multiplicity free.

\noindent \textbf{Case 5}: $h_k=0$ with $h_{k+1} + \cdots + h_{b_L} \leq 2$. We once again use Corollary \ref{c:headCriterionMain}\eqref{c:headCriterionMain3} to see that the boxes in $\tab_{\underline{\theta}}$ with values greater than those in $\Bl_{L,k}$ can only appear in row $d-1$ or $d$. However, since $h_k=0$ we also have $h_{k} + \cdots + h_{b_L} \leq 2$. Thus Corollary \ref{c:headCriterionMain}\eqref{c:headCriterionMain3} also gives us that values in $\Bl_{L,k}$ can only appear in row $d-1$ or $d$. Hence, if a box has a value in $\Bl_{L,k}$, then the box must be in row $d-1$ or $d$. Thus a column that has a box with a value in $\Bl_{L,k}$ can only have a box in the $d$th row with a value greater than those in $\Bl_{L,k}$. This combines with the fact that $\tab_{\underline{\theta}}$ is semistandard and the definition of the skew diagrams to imply that the $m^n$-complement $(\lambda)^\# = (p)$ for some positive integer $p$ or  $(\lambda)^\# = \emptyset$. This gives us our desired result as in the previous case. 

\noindent $(\Rightarrow)$ We will prove the contrapositive. That is, suppose there is a $k$ where $N_k \neq 1$, $h_1 + \cdots + h_{k-1} + 1 < N_1 + \cdots + N_{k-1}$ or if $h_k = N_k$ then $h_1 + \cdots + h_{k-1} + 2 < N_1 + \cdots + N_{k-1}$, and either $h_k=0$ with $h_{k+1} + \cdots + h_{b_L} > 2$ or $h_k>0$ with $h_{k+1} + \cdots + h_{b_L} \geq 2$. Set $\lambda = \lambda_{\underline{\theta}}^{(k)}$ and $\mu = \mu_{\underline{\theta}}^{(k)}$. Our goal will be the exhibit a standard degree 3 head $\underline{\theta}$ such that $\mathbb{W}^{\lambda / \mu}$ is not a multiplicity free $\GL_{N_k}$-module. 

In the case where $h_k=0$ let $s \geq k+1$ be the minimum index such that $h_s \neq 0$; such an index must exist by our hypothesis. We define the non-negative integers $m_1,\ldots,m_{b_L}$ by setting $m_k=1$, $m_s = h_s - 1$, and $m_i = h_i$ for all other indices. In the case where $N_k > h_k>0$ we simply set $m_i = h_i$ for all indices. In the case where $h_k=N_k$ let $t < k$ be the maximum index such that $h_t < N_t$. We define the non-negative integers $m_1,\ldots,m_{b_L}$ by setting $m_k=h_k-1$, $m_t = h_t + 1$, and $m_i = h_i$ for all other indices.  In all three cases, the integers $m_1,\ldots,m_{b_L}$ satisfy the conditions of Lemma~\ref{l:headCriterion}. Further, in all three cases,
\begin{equation}
\label{e:mc1Neww}
m_{k+1} + \cdots + m_{b_L} \geq 2, m_1 + \cdots + m_{k-1} + 1 < N_1 + \cdots + N_{k-1}\textrm{, and }N_k > m_k > 0.
\end{equation}

Let $p < k$ be the maximal index such that $m_k < N_k$ and let $q \geq k+1$ be the minimal index such that $m_q \neq 0$; such indices must exist by \eqref{e:mc1Neww}.  We have four possible cases.

\noindent \textbf{Case 1}: $p \neq 1$ and $q < b_L$. Then we may define three degree 1 heads.
\begin{center}
$\begin{array}{r@{\hspace{4pt}}l}
\theta_1=& \Theta(m_1 ,\ldots , m_{b_L}) \\[4pt]
\theta_2=& \Theta(m_1,\ldots,m_{p}+1,\ldots,m_k,\ldots,m_{q} - 1,\ldots,m_{b_L}) \\[4pt]
\theta_3=& \Theta(m_1+1,\ldots,m_{p}+1,\ldots,m_k,\ldots,m_{q} - 1,\ldots,m_{b_L}-1) \\[4pt]
\end{array}$
\end{center}
These are easily verified to satisfy the conditions of Lemma~\ref{l:headCriterion} since $m_1 = h_1 < N_1$ and $m_{b_L}=h_{b_L}>0$ by \eqref{e:hReduced}. Further $\theta_1 \geq \theta_2 \geq \theta_3$, and so $\underline{\theta}:=(\theta_1,\theta_2,\theta_3)$ is a standard degree 3 head. A careful analysis of $\mathbb{W}_{\underline{\theta}}$ reveals that $\lambda / \mu = (3^{m_k},2,1) / (2,1)$ with $N_k > m_k > 0$. For $\nu = (3^{m_k - 1}, 2, 1)$, we have by Lemma~\ref{lemma:LRSphericalClassH}\eqref{lemma:LRSphericalClassH4} that $\lr{\mu}{\nu}{\lambda}=2$. Since the length of $\nu$ is $m_k+1 \leq N_k$ we have that $\sch{\lambda / \mu}(x_1,\ldots,x_{N_k})$ is not multiplicity free, and hence $\mathbb{W}^{\lambda / \mu}$ is not multiplicity free.

\noindent \textbf{Case 2}: $p = 1$ and $q = b_L$. Note that this implies, by the maximality of $p$ and \eqref{e:mc1Neww}, that $m_1 + 1 < N_1$. It also implies, by the minimality of $q$ and \eqref{e:mc1Neww}, that $m_{b_L} \geq 2$. Then we may define three degree 1 heads.
\begin{center}
$\begin{array}{r@{\hspace{4pt}}l}
\theta_1=& \Theta(m_1 ,\ldots , m_{b_L}) \\[4pt]
\theta_2=& \Theta(m_1+1,\ldots,m_k,\ldots,m_{b_L}-1) \\[4pt]
\theta_3=& \Theta(m_1+2,\ldots,m_k,\ldots,m_{b_L}-2) \\[4pt]
\end{array}$
\end{center}
As in the previous case these are easily verified to be degree 1 heads with $\theta_1 \geq \theta_2 \geq \theta_3$, and so $\underline{\theta}:=(\theta_1,\theta_2,\theta_3)$ is a standard degree 3 head. Once again $\lambda / \mu = (3^{m_k},2,1) / (2,1)$ and thus $\mathbb{W}^{\lambda / \mu}$ is not multiplicity free.

\noindent \textbf{Case 3}: $p = 1$ and $q < b_L$. We define three degree 1 heads.
\begin{center}
$\begin{array}{r@{\hspace{4pt}}l}
\theta_1=& \Theta(m_1 ,\ldots , m_{b_L}) \\[4pt]
\theta_2=& \Theta(m_1+1,\ldots,m_k,\ldots,m_{q} - 1,\ldots,m_{b_L}) \\[4pt]
\theta_3=& \Theta(m_1+2,\ldots,m_k,\ldots,m_{q} - 1,\ldots,m_{b_L}-1) \\[4pt]
\end{array}$
\end{center}
Then the standard degree 3 head $\underline{\theta}:=(\theta_1,\theta_2,\theta_3)$ has associated skew diagram $\lambda / \mu= (3^{m_k},2,1) / (2,1)$. Thus $\mathbb{W}^{\lambda / \mu}$ is not multiplicity free.

\noindent \textbf{Case 4}: $p > 1$ and $q = b_L$. The three degree 1 heads in this case will be the following.
\begin{center}
$\begin{array}{r@{\hspace{4pt}}l}
\theta_1=& \Theta(m_1 ,\ldots , m_{b_L}) \\[4pt]
\theta_2=& \Theta(m_1,\ldots,m_{p}+1,\ldots,m_k,\ldots,m_{b_L}-1) \\[4pt]
\theta_3=& \Theta(m_1+1,\ldots,m_{p}+1,\ldots,m_k,\ldots,m_{b_L}-2) \\[4pt]
\end{array}$
\end{center}
Once again for $\underline{\theta}:=(\theta_1,\theta_2,\theta_3)$, the associated skew diagram is $\lambda / \mu= (3^{m_k},2,1) / (2,1)$ indicating $\mathbb{W}^{\lambda / \mu}$ is not multiplicity free.

Thus in all four possible cases criterion $\MC$ is not satisfied.
\end{proof}

We are now ready to use the two previous propositions to prove our main theorem. Fortunately, the conditions may be stated in a simpler manner when both are required to hold.

\begin{theorem}
\label{t:mainSphericalClassification}
The stable, reduced Levi-Schubert quadruple $(w,d,N,L)$ is multiplicity free (equivalently spherical) if and only if one of the following holds

\begin{enumerate}[label=(\roman*)]
\item $b_L \leq 2$
\item $b_L = 3$, and at least one of $N_2 = 1$, $h_1 + 1 \geq N_1$, $N_2=h_2$ with $h_1 + 2 \geq N_1$, $h_2 > 0$ with $h_3 < 2$, $h_2 = 0$ with $h_3 \leq 2$ holds
\item $b_L \geq 4$, $p_w = 2$ or if $p_w > 2$, then $h_1 + \cdots + h_{p_w - 1} + 1 \geq N_1 + \cdots + N_{p_w - 1}$
\end{enumerate}
where $1 < p_w < b_L - 1$ is the minimum index such that $h_{p_w + 1} + \cdots + h_{b_L} < 2$. Such an index may not exist, if it does not set $p_w = b_L - 1$.
\end{theorem}
\begin{proof}
We will prove the above in three cases depending on the value of $b_L$.

\noindent \textbf{Case 1}: $b_L \leq 2$. In this case, the two sets of conditions from Proposition~\ref{p:MCSatCrit} and Proposition~\ref{p:MCCSatCrit} will always be vacuously true, and hence  $(w,d,N,L)$ will always be multiplicity free.

\noindent \textbf{Case 2}: $b_L = 3$. In this case the conditions from Proposition~\ref{p:MCCSatCrit} are vacuously true. The conditions from Proposition~\ref{p:MCSatCrit} for $k=2$ are precisely that at least one of $N_2 = 1$, $h_1 + 1 \geq N_1$, $N_2=h_2$ with $h_1 + 2 \geq N_1$, $h_2 > 0$ with $h_3 < 2$, $h_2 = 0$ with $h_3 \leq 2$ holds.

\noindent \textbf{Case 3}: $b_L \geq 4$. We start with the case where $p_w < b_L - 1$. Then we have that $h_{k + 1} + \cdots + h_{b_L} < 2$ for all $p_w \leq k < b_L$. Thus the conditions for Proposition~\ref{p:MCSatCrit} and Proposition~\ref{p:MCCSatCrit} are always satisfied for such $k$. If $p_w = 2$, we are done. Otherwise, if $p_w > 2$, then $h_1 + \cdots + h_{p_w - 1} + 1 \geq N_1 + \cdots N_{p_w - 1}$ implies that $h_1 = N_1 + 1$, $h_2 = N_2$, \ldots, $h_{p_w - 1} = N_{p_w - 1}$. This means that $h_1 + \cdots + h_k + 1 \geq N_1 + \cdots + N_k$ for all $1 \leq k \leq p_w - 1$. Hence the conditions for Proposition~\ref{p:MCSatCrit} and Proposition~\ref{p:MCCSatCrit} are always satisfied for such $k$. Thus the conditions of these two propositions are always satisfied.

In the case where $p_w = b_L - 1$, we can see that $h_1 + \cdots + h_k + 1 \geq N_1 + \cdots + N_k$ for all $1 \leq k \leq b_L - 2$. This precisely means that Proposition~\ref{p:MCSatCrit} is satisfied for $1 \leq k \leq b_L - 1$ and Proposition~\ref{p:MCCSatCrit} is satisfied for $1 \leq k \leq b_L - 2$. Thus they are always satisfied. Hence Proposition~\ref{p:MultFreeHighestLevel} gives us that $(w,d,N,L)$ is multiplicity free.

For the other direction, we assume that $(w,d,N,L)$ is multiplicity free. Criterion $\MC$ and $\MCC$ are thus always satisfied. If $p_w > 2$ then we know that for $p_w - 1$ one of the conditions from Proposition~\ref{p:MCCSatCrit} must hold. The minimality of $p_w$ indicates that it can not be that  $h_{p_w} + \cdots + h_{b_L} < 2$, and hence it must be that $h_1 + \cdots + h_{p_w - 1} + 1 \geq N_1 + \cdots + N_{p_w - 1}$.
\end{proof}

When $L_{max}$ is the maximal Levi subgroup acting on $X(w)$ we may further simplify the our result.
\begin{corollary}
\label{c:mainSphericalClassification}
The stable, reduced Levi-Schubert quadruple $(w,d,N,L_{max})$ where $L_{max}$ is the maximal Levi acting on $X(w)$ is multiplicity free (equivalently spherical) if and only if one of the following holds
\begin{enumerate}[label=(\roman*)]
\item $b_{L_{max}} \leq 2$
\item $b_{L_{max}} = 3$, and at least one of $h_1 + 1 = N_1$ or $h_3 = 1$ holds.
\end{enumerate}
\end{corollary}
\begin{proof}
Since $L_{max}$ is the maximal Levi which acts on $X(w)$, we know that \eqref{e:hFormMax} holds. In the case where $b_{L_{max}} = 3$, this implies that the only conditions from Theorem~\ref{t:mainSphericalClassification} that can hold are $h_1 + 1 \geq N_1$ or $h_2 > 0$ with $h_3 < 2$. Using \eqref{e:hFormMax} we can further reduce these conditions to $h_1 + 1 = N_1$ or $h_3 = 1$. In the case where $b_{L_{max}} \geq 4$ we see that \eqref{e:hFormMax} implies $p_w=b_{L_{max}} - 1 > 2$. But then $h_1 + \cdots + h_{p_w - 1} + 1 < N_1 + \cdots + N_{p_w - 1}$ since each $h_k < N_k$. Thus, in this case, the conditions from Theorem~\ref{t:mainSphericalClassification} are never satisfied.
\end{proof}

\section{Toric Schubert varieties in the Grassmannian}
\label{sec:toric}

Many mathematicians have been interested in Toric degenerations of Schubert varieties. The second author and N. Gonciulea \cite{MR1417711} gave toric degenerations of Schubert varities in a miniscule $G/P$ and for certain Schubert varieties in $\SL_N / B$. Subsequently, this work was extended to all Schubert varieties in $\SL_N / B$ by R. Dehy and R.W.T Yu in \cite{MR1870638}. Building on these works, in \cite{MR1888475}, P. Caldero gave toric degenerations for Schubert varieties in any $G/P$. A natural related question is which Schubert varieties are themselves toric varieties. This has been answered entirely for Schubert varieties in $G/B$ by P. Karuppuchamy in \cite{MR3044412}. Along these lines, using Theorem~\ref{t:mainSphericalClassification} we get a description of a class of Schubert varieties in the Grassmannian that are toric varieties for a quotient of the maximal torus by a subtorus.
 
Note that if a stable, reduced Levi-Schubert quadruple $(w,d,N,T)$ is spherical then $X(w)$ is a toric variety. This follows from the fact that we have an open, dense $T$-orbit. Taking a point in this orbit and letting $H$ be the isotropy subgroup of $T$ at this point, we identify the dense $T$-orbit with $T/H$ and $X(w) = \overline{T / H}$.

\begin{proposition}
Let $(w,d,N,T)$ be a stable Levi-Schubert quadruple with reduction $(\overline{w},\overline{d},\overline{N},\overline{T})$. Then $(w,d,N,T)$ is spherical if and only if $\overline{w} = (2,\ldots,\overline{d},\overline{N})$ or $\overline{w} = (\overline{N})$.
\end{proposition} 
\begin{proof}
Note that for any $\overline{T}$ we have $b_{\overline{T}}=\overline{N}$ and $\Bl_{\overline{T},k}=\{ k \}$ for $1 \leq k \leq \overline{N}$. Recall that since $(\overline{w},\overline{d},\overline{N},\overline{T})$ is reduced $\overline{w}$ does not have its first entry equal to 1 and its last entry is equal to $\overline{N}$.

By Corollary~\ref{c:mainSphericalReduction} $(w,d,N,T)$ is spherical if and only if $(\overline{w},\overline{d},\overline{N},\overline{T})$ is spherical. We will show that $\overline{w}$ can satisfy the conditions of Theorem~\ref{t:mainSphericalClassification} if and only if it is of the form $(2,\ldots,\overline{d},\overline{N})$ or $(\overline{N})$. We first consider the case where $\overline{N}=2$. Then $\overline{d}=1$. Further, the fact that $(\overline{w},\overline{d},\overline{N},\overline{T})$ is reduced implies that $\overline{w}=(2)$. This $\overline{w}$ is of the form stated in the hypothesis and $(\overline{w},\overline{d},\overline{N},\overline{T})$ is spherical since $b_{\overline{T}}=2$. When $\overline{N}=3$, $\overline{d}$ is either 1 or 2. When $\overline{d}$ is 1, then $\overline{w}$ can only be $(3)$. Then $b_{\overline{T}}=3$ and $h_2=0$ with $h_3=1$. Thus $(\overline{w},\overline{d},\overline{N},\overline{T})$ is spherical with $\overline{w}$ of the form stated in the hypothesis. When $\overline{d}$ is 2, then $\overline{w}$ must be $(2,3)$. Then $b_{\overline{T}}=3$ and $h_2=1$ with $h_3=1$. Once again $(\overline{w},\overline{d},\overline{N},\overline{T})$ is spherical with $\overline{w}$ of the correct form.

We now conclude by considering $\overline{N}\geq4$. Then $b_{\overline{T}}\geq 4$ and Theorem~\ref{t:mainSphericalClassification} gives that $(\overline{w},\overline{d},\overline{N},\overline{T})$ will be spherical if and only if $p_w = 2$ or if $p_w > 2$, then $h_1 + \cdots + h_{p_w - 1} + 1 \geq N_1 + \cdots + N_{p_w - 1}$. The only way that $p_w$ can equal $2$ is if $\overline{d}$ equals 1 or 2 and $\overline{w}=(\overline{N})$ or $\overline{w}=(2,\overline{N})$ respectively. When $p_w > 2$, the only way that $\overline{w}$ can satisfy the condition $h_1 + \cdots + h_{p_w - 1} + 1 \geq N_1 + \cdots + N_{p_w - 1}$ is if $\overline{w} = (2,\ldots,p_w,\overline{N})=(2,\ldots,\overline{d},\overline{N})$.
\end{proof}

\begin{corollary}
\label{c:toricSchubert}
The Schubert variety $X(w)$ is a toric variety for a quotient of the maximal torus $T$ if in the reduction $(\overline{w},\overline{d},\overline{N},\overline{T})$ we have that $\overline{w} = (2,\ldots,\overline{d},\overline{N})$ or $w=(\overline{N})$. This is equivalent to $w$ being one of two possible forms
\begin{enumerate}[label=(\arabic*)]
\item $(1,\ldots,p,p+2,\ldots,d,f)$ for some integers $0 \leq p < d-1$ and $d < f \leq N$
\item $(1,\ldots,d-1,f)$ for some integer $d < f \leq N$
\end{enumerate}
\end{corollary}

\bibliography{Classification.Spherical.RHodges.VLakshmibai}

\begin{thebibliography}{{Cup}09}

\bibitem[BH18]{2018arXiv180305515B}
M.~{Bilen Can} and R.~{Hodges}.
\newblock {Smooth Schubert varieties are spherical}.
\newblock {\em ArXiv e-prints}, March 2018.

\bibitem[BHL18]{2018arXiv180704879B}
M.~{Bilen Can}, R.~{Hodges}, and V.~{Lakshmibai}.
\newblock {Toroidal Schubert Varieties}.
\newblock {\em ArXiv e-prints}, July 2018.

\bibitem[Bor91]{MR1102012}
Armand Borel.
\newblock {\em Linear algebraic groups}, volume 126 of {\em Graduate Texts in
  Mathematics}.
\newblock Springer-Verlag, New York, second edition, 1991.

\bibitem[BP14]{MR3198836}
P.~Bravi and G.~Pezzini.
\newblock Wonderful subgroups of reductive groups and spherical systems.
\newblock {\em J. Algebra}, 409:101--147, 2014.

\bibitem[BP16]{MR3473657}
P.~Bravi and G.~Pezzini.
\newblock Primitive wonderful varieties.
\newblock {\em Math. Z.}, 282(3-4):1067--1096, 2016.

\bibitem[Cal02]{MR1888475}
Philippe Caldero.
\newblock Toric degenerations of {S}chubert varieties.
\newblock {\em Transform. Groups}, 7(1):51--60, 2002.

\bibitem[{Cup}09]{2009arXiv0907.2852C}
S.~{Cupit-Foutou}.
\newblock {Wonderful Varieties: A geometrical realization}.
\newblock {\em ArXiv e-prints}, July 2009.

\bibitem[DY01]{MR1870638}
R.~Dehy and R.~W.~T. Yu.
\newblock Degeneration of {S}chubert varieties of {${\rm SL}_n/B$} to toric
  varieties.
\newblock {\em Ann. Inst. Fourier (Grenoble)}, 51(6):1525--1538, 2001.

\bibitem[FH91]{MR1153249}
William Fulton and Joe Harris.
\newblock {\em Representation theory}, volume 129 of {\em Graduate Texts in
  Mathematics}.
\newblock Springer-Verlag, New York, 1991.
\newblock A first course, Readings in Mathematics.

\bibitem[Ful97]{MR1464693}
William Fulton.
\newblock {\em Young tableaux}, volume~35 of {\em London Mathematical Society
  Student Texts}.
\newblock Cambridge University Press, Cambridge, 1997.
\newblock With applications to representation theory and geometry.

\bibitem[GL96]{MR1417711}
N.~Gonciulea and V.~Lakshmibai.
\newblock Degenerations of flag and {S}chubert varieties to toric varieties.
\newblock {\em Transform. Groups}, 1(3):215--248, 1996.

\bibitem[Gut10]{MR2737323}
Christian Gutschwager.
\newblock On multiplicity-free skew characters and the {S}chubert calculus.
\newblock {\em Ann. Comb.}, 14(3):339--353, 2010.

\bibitem[HL18]{HodgesLakshmibai}
Reuven Hodges and Venkatramani Lakshmibai.
\newblock Levi subgroup actions on schubert varieties, induced decompositions
  of their coordinate rings, and sphericity consequences.
\newblock {\em Algebras and Representation Theory}, 2018.
\newblock https://doi.org/10.1007/s10468-017-9744-6.

\bibitem[Kar13]{MR3044412}
Paramasamy Karuppuchamy.
\newblock On {S}chubert varieties.
\newblock {\em Comm. Algebra}, 41(4):1365--1368, 2013.

\bibitem[LB15]{MR3408060}
V.~Lakshmibai and Justin Brown.
\newblock {\em The {G}rassmannian variety}, volume~42 of {\em Developments in
  Mathematics}.
\newblock Springer, New York, 2015.
\newblock Geometric and representation-theoretic aspects.

\bibitem[Los09]{MR2495078}
Ivan~V. Losev.
\newblock Uniqueness property for spherical homogeneous spaces.
\newblock {\em Duke Math. J.}, 147(2):315--343, 2009.

\bibitem[Lun01]{MR1896179}
D.~Luna.
\newblock Vari\'et\'es sph\'eriques de type {$A$}.
\newblock {\em Publ. Math. Inst. Hautes \'Etudes Sci.}, (94):161--226, 2001.

\bibitem[LV83]{MR705534}
D.~Luna and Th. Vust.
\newblock Plongements d'espaces homog\`enes.
\newblock {\em Comment. Math. Helv.}, 58(2):186--245, 1983.

\bibitem[Per14]{MR3177371}
Nicolas Perrin.
\newblock On the geometry of spherical varieties.
\newblock {\em Transform. Groups}, 19(1):171--223, 2014.

\bibitem[Sta99]{MR1676282}
Richard~P. Stanley.
\newblock {\em Enumerative combinatorics. {V}ol. 2}, volume~62 of {\em
  Cambridge Studies in Advanced Mathematics}.
\newblock Cambridge University Press, Cambridge, 1999.
\newblock With a foreword by Gian-Carlo Rota and appendix 1 by Sergey Fomin.

\bibitem[TY10]{MR2583223}
Hugh Thomas and Alexander Yong.
\newblock Multiplicity-free {S}chubert calculus.
\newblock {\em Canad. Math. Bull.}, 53(1):171--186, 2010.

\end{thebibliography}

\end{document}